\documentclass[]{article}
\usepackage{amssymb,amsmath,amsthm,bm,mathrsfs,paralist}
\usepackage{url,color,units,graphicx,mathtools,bbold}
\usepackage[T1]{fontenc}
\usepackage[margin=1in]{geometry}
\RequirePackage[colorlinks,citecolor=blue,urlcolor=blue]{hyperref}
\usepackage{pdfsync}

\DeclareMathOperator{\Var}{Var}
\DeclareMathOperator{\Cov}{Cov}

\DeclareMathOperator{\ind}{\text{\sc{ind}}}
\DeclareMathOperator{\indL}{\underline{\ind}}
\DeclareMathOperator{\indH}{\overline{\ind}}

\DeclareMathOperator{\sinc}{sinc}
\DeclareMathOperator{\F}{\mathscr{F}}
\newcommand{\<}{\langle}
\renewcommand{\>}{\rangle}

\newcommand{\N}{\mathbb{N}}
\newcommand{\Z}{\mathbb{Z}}

\newcommand{\R}{\mathbb{R}}
\newcommand{\C}{\mathbb{C}}

\newcommand{\cS}{\mathcal{S}}

\newcommand{\rL}{\mathcal{L}}
\newcommand{\rf}{F}
\newcommand{\ba}{\mathbf{a}}

\newcommand{\bA}{\mathbf{A}}
\newcommand{\bQ}{\mathbf{Q}}

\newcommand{\cI}{\mathcal{I}}

\newcommand{\cK}{\mathcal{K}}

\renewcommand{\P}{\mathrm{P}}
\newcommand{\E}{\mathrm{E}}

\newcommand{\1}{\mathbb{1}}
\renewcommand{\d}{{\rm d}}
\renewcommand{\Re}{\mathrm{Re}}
\renewcommand{\Im}{\mathrm{Im}}
\newcommand{\e}{{\rm e}}

\renewcommand{\leq}{\leqslant}
\renewcommand{\ge}{\geqslant}
\renewcommand{\le}{\leqslant}

\author{Davar Khoshnevisan\\University of Utah
	\and Marta Sanz-Sol\'e\\Universitat de Barcelona}
\title{\bf Optimal regularity of SPDEs with additive noise\thanks{Research 
supported in part by the United States' National Science Foundation grant  DMS-1855439, and by the Spanish {\em Ministerio de Ciencia e Innovaci\'on} PID2020-118339GB-I00
}}
\newtheorem{stat}{Statement}[section]
\newtheorem{proposition}[stat]{Proposition}

\newtheorem{theorem}[stat]{Theorem}
\newtheorem{lemma}[stat]{Lemma}
\theoremstyle{definition}\newtheorem{definition}[stat]{Definition}

\newtheorem{remark}[stat]{Remark}

\newtheorem{example}[stat]{Example}

\numberwithin{equation}{section}

\begin{document}
\maketitle
\begin{abstract} 
	The sample-function regularity of the random-field solution to a stochastic 
	partial differential equation (SPDE) depends naturally on the roughness of
	the external noise, as well as on the properties of the underlying integro-differential operator 
	that is used to define the equation. In this paper, we consider parabolic and hyperbolic SPDEs 
	on $(0\,,\infty)\times\R^d$ of the form 
	\[
		\partial_t u = \rL u + g(u) + \dot{\rf}
		\qquad\text{and}\qquad
		\partial^2_t u = \rL u + c + \dot{\rf},
	\]
	with suitable initial data,  forced with a
	space-time homogeneous Gaussian noise $\dot{\rf}$
	that is white in its time variable and correlated in its space variable, and 
	driven by the generator  $\rL$ of a genuinely $d$-dimensional L\'evy process $X$.
	We find optimal H\"older  conditions for the respective random-field solutions to these
	SPDEs. Our conditions are stated in terms of indices that describe thresholds on the 
	 integrability of some functionals of the characteristic exponent of the process $X$
	 with respect to the spectral measure of the spatial covariance of $\dot\rf$. 
	 Those indices are suggested by references \cite{SanzSoleSarra2000,SanzSoleSarra2002} 
	 on the particular case that $\rL$ is the Laplace operator on $\R^d$.
	\\[.1in]
\end{abstract}

\noindent{\it Key words and phrases.} Stochastic partial differential equation, 
Gaussian noise, L\'evy process, characteristic exponent, optimal H\"older regularity.\\

 \noindent{\it 2010 Mathematics Subject Classification.}
	Primary: 60H15, 60G51, 60G60;  Secondary: 35R60, 35E05.

\section{Introduction}
\label{s1}

We consider a semilinear stochastic partial differential equation (SPDE) of the following parabolic type,
\begin{equation}\left[\label{SHE}\begin{split}
	&\partial_t u = \rL u + g(u) + \dot{\rf}
		\quad\text{on $(0\,,\infty)\times\R^d$},\\
	&\text{subject to}\quad u(0)=u_0\quad\text{ on $\R^d$},
\end{split}\right.\end{equation}
as well as a hyperbolic stochastic PDE of the form,	
\begin{equation}\left[\label{SWE}\begin{split}
	&\partial^2_t u = \rL u + c + \dot{\rf}
		\quad\text{on $(0\,,\infty)\times\R^d$},\\
	&\text{subject to}\quad u(0)=u_0
		\quad\text{and}\quad 
		\left.\partial_tu(t) \right|_{t=0} = v_0\quad\text{ on $\R^d$},
\end{split}\right.\end{equation}
where  $\rL$ denotes the generator of a genuinely $d$-dimensional
L\'evy process (see \eqref{gen}),
$g:\R\to\R$, $c\in \R$, and
$u_0:\R^d\to\R$ and $v_0:(0\,,\infty)\times\R^d\to\R$ are deterministic functions.
The quantity $\dot{\rf}$ is a centered space-time Gaussian noise whose covariance is, somewhat
informally, defined as
\begin{equation}\label{Cov}
	\Cov[\dot{\rf}(t_1\,,x) \,,\dot{\rf}(t_2\,,y)] = \delta_0(t_1-t_2)\Gamma(x-y)
	\quad\text{for every $t_1,t_2\ge 0$ and $x,y\in\R^d$,}
\end{equation}
where $\Gamma$ is a tempered, nonnegative-definite Borel measure on $\R^d$.
We can understand $\dot{\rf}$ more carefully via its action on a
rapidly decreasing test function $\phi\in\cS(\R^{1+d})$ as follows: Let
\[
	\dot{\rf}(t\,,\phi) = \int_{(0,t)\times\R^d}\phi(s\,,x)\, \rf(\d s,\d x)
	\qquad[t>0],
\]
where the stochastic integral is a Wiener integral normalized to ensure that
\[
	\Cov[ \dot{\rf}(t_1\,,\phi_1)\,,\dot{\rf}(t_2\,,\phi_2)] = \int_0^{t_1\wedge t_2}\d s\int_{\R^d}
	\Gamma(\d x)\left(\phi_1(s) * \tilde\phi_2(s)\right)
	\hskip.5in\text{for every $\phi_1,\phi_2\in\cS(\R^d)$.}
\]
In this formula, the symbol ``$*$'' denotes the convolution operator in the spatial variable, and
$\tilde\phi(t\,,x) = \phi(t\,,-x)$ defines the reflection of $\phi:\R_+\times\R^d\to\R$ in 
its space variable.

Let   ``$\hat{\phantom{a}}\,$''  denote the  Fourier transform 
on $\R^d$, normalized so that 
\begin{equation}\label{norm-fourier}
	\hat{f}(\xi) = \int_{\R^d}\e^{ix\cdot\xi} f(x)\,\d x
	\qquad\text{for all $f\in L^1(\R^d)$ and $\xi\in\R^d$}.
\end{equation}
We can then let $\mu=\hat{\Gamma}$ and $\psi=-\hat{\rL}$
respectively denote the spectral measure for the spatial aspect
of the noise and the characteristic exponent of the underlying
L\'evy process (see \S\ref{sec:Levy}).
Assume that $g:\R\rightarrow \R$ is Lipschitz continuous.  An extension of the theory of Dalang \cite{Dalang} (see also Brz\'ezniak and van 
Neerven \cite{BvN} and Khoshnevisan and Kim \cite{KK})
ensures that for every $d\ge 1$ (respectively, $d\in\{1,2,3\}$) the SPDE \eqref{SHE} (respectively, \eqref{SWE}) has a random-field solution if
\begin{equation}\label{Dalang}
	 \int_{\R^d}\frac{\mu(\d\xi)}{1+\Re\psi(\xi)}<\infty.
\end{equation}
Moreover, this condition is optimal in the sense that it is
necessary as well as sufficient when $g$ is a constant function. 
Therefore, throughout this paper we assume that Dalang's condition
\eqref{Dalang} holds in order to ensure that \eqref{SHE} and \eqref{SWE} are well posed.

The principal aim of this paper is to establish optimal H\"older regularity conditions
for the random-field solution to \eqref{SHE} and \eqref{SWE}, respectively. 
For the stochastic heat equation \eqref{SHE}, we 
restrict attention to  random initial data $u_0$ that is 
independent of the noise $\dot{F}$ and satisfies moment-type conditions;
see Proposition \ref{pr:heat} for details. In the case of the wave equation \eqref{SWE}, 
we will consider null initial data, although in spatial dimension $d\in\{1,2,3\}$ it would be possible to consider non-null smooth initial data as well \cite{DalangSanzSole-2,MilletSanzSole}.
The next two theorems contain the main findings of this paper.

\begin{theorem}\label{th:heat}
	Let $\{u(t\,,x);\ (t\,,x)\in \R_+\times \R^d\}$ denote the random-field solution to 
	the parabolic SPDE $\eqref{SHE}$. 
	\begin{compactenum}[\rm (a)]
		\item For every $x\in\R^d$, the random function $t\mapsto u(t\,,x)$
			is a.s.\ locally H\"older continuous provided that
			\begin{equation}\label{cond:H}
				\int_{\R^d} \frac{|\psi(\xi)|^a\,\mu(\d \xi)}{1+\Re\psi(\xi)}<\infty
				\qquad\text{for some $a\in(0\,,1)$}.
			\end{equation}
		\item If $g$ is a constant and \eqref{cond:H} fails, then 
			$t\mapsto u(t\,,x)$ is a.s.\ not locally H\"older continuous for every  $x\in\R^d$.
		\item The function $(t\,,x)\mapsto u(t\,,x)$ is a.s.\ locally H\"older continuous provided that
			\begin{equation}\label{cond:R}
				\int_{\R^d}\frac{\|\xi\|^{2b}\,\mu(\d \xi)}{1+\Re\psi(\xi)}<\infty
				\qquad\text{for some $b\in(0\,,1)$}.
			\end{equation}
		\item If $g$ is a constant and \eqref{cond:R} fails, then
			$x\mapsto u(t\,,x)$ is a.s.\ not locally H\"older continuous  for every $t>0$.
	\end{compactenum}
\end{theorem}

When $g$ is constant and the initial condition $u_0$ vanishes,
 assertions (a) and (c) of Theorem \ref{th:heat} follow respectively  from \eqref{H:x} and \eqref{H:t} 
in Theorem \ref{th:Holder:H}. In the constant-$g$ case, parts (b) and (d) 
of Theorem \ref{th:heat} follow from assertion 2. of Theorem \ref{th:Holder:H}. 
Theorem \ref{th:heat-nonlinear} below yields  (a) and (c) in the more general 
case that $g$ is non constant and the initial condition $u_0$ satisfies the hypotheses 
of Proposition \ref{pr:heat}. We mention also 
that Theorem \ref{th:Holder:H} and Theorem \ref{th:heat-nonlinear} provide bounds 
on the optimal H\"older indices of the processes in question.

The preceding is a presentation of optimal H\"older regularity for parabolic SPDEs.
The following is a hyperbolic counterpart of these results.

\begin{theorem}\label{th:wave}
	Let $\{u(t\,,x);\ (t\,,x)\in \R_+\times \R^d\}$ 
	denote the random-field solution to $\eqref{SWE}$, and 
	assume that $g$ is a constant function and the initial functions $u(0)$ and $\partial_t u(0+\,,\cdot)$
	both vanish. Then,
	\begin{compactenum}[\rm (i)]
		\item If \eqref{cond:H} holds then for, every $x\in\R^d$, the random function $t\mapsto u(t\,,x)$
			is a.s.\ locally H\"older continuous.
		\item If condition \eqref{cond:H} fails then for every $x\in\R^d$, 
			$t\mapsto u(t\,,x)$ is a.s.\ not locally H\"older continuous.
		\item The function $(t\,,x)\mapsto u(t\,,x)$ is a.s.\ locally H\"older continuous provided that \eqref{cond:R} holds.
		\item If condition \eqref{cond:R} fails, then for every $t>0$,
			$x\mapsto u(t\,,x)$ is a.s.\ not locally H\"older continuous.
	\end{compactenum}
\end{theorem}

These results are proved in Theorem \ref{th:Holder:W}, where  optimal H\"older indices 
are also given. In particular, parts (i) and (iii) follow from \eqref{Holder:W-1}, \eqref{Holder:W-2}, respectively, while the assertions (ii) and (iv) follow from part 2 of Theorem \ref{th:Holder:W}.
In the last part of Section \ref{s6}, we discuss briefly why an extension of Theorem \ref{th:wave} to a non constant $g$ and (or) non vanishing initial conditions seems to be at present out of reach.

\begin{remark}\label{rem:heat}\begin{compactenum}
	
\item Theorems \ref{th:heat} and \ref{th:wave} sometimes have relations to SPDEs driven by
		fractional powers of $\rL$; see Section \ref{sec:frac}.
	\item One can read from Theorems \ref{th:heat} and \ref{th:wave} that condition
		\eqref{cond:R} always implies \eqref{cond:H}. This implication
		can be more directly read off from the following classical estimate,
		which is itself a ready consequence of the L\'evy-Khintchine formula
		\eqref{LKF} below:
		\begin{equation}\label{Bochner}
			\sup_{\xi\in\R^d}\frac{|\psi(\xi)|}{1+\|\xi\|^2}<\infty.
		\end{equation}
		There can be no converse implication. Indeed,
		Proposition \ref{pr:ind:-} yields a non-trivial example where \eqref{cond:H}
		holds but \eqref{cond:R} does not; see also Example \ref{ex:CP}. In
		this way, we can see from Theorems \ref{th:heat} and \ref{th:wave}
		that optimal H\"older regularity in the space variable implies optimal H\"older regularity
		in the time variable, but not conversely.
	\item It should be possible to combine the methods
		of this paper and those of Jacob \cite{Jacob} and Schilling \cite{Schilling} to extend
		our results to the case where $\rL$ is the generator of a nice,  L\'evy like, Markov process.
		Though we have not  considered such a generalization, it would be a worthwhile endeavor
		since a celebrated theorem of \c{C}inlar, Jacod, Protter, and Sharpe \cite{CJPS} asserts that
		every nice (quasi continuous, to be sure) Markov process that is a semimartingale is 
		a time-change of a locally L\'evy-like process. 
\end{compactenum}\end{remark}

Parts (b) and (d) of Theorem \ref{th:heat}, and (ii) and (iv) of Theorem \ref{th:wave} are the main contributions of those results. 
Consider the case that $\rL=\Delta$, the Laplace operator. Then 
$\psi(\xi)=\|\xi\|^2$, and Conditions \eqref{cond:H} and \eqref{cond:R} coincide
and are both equivalent to 
\begin{equation}\label{cc}
	\int_{\R^d} \frac{\mu(\d x)}{(1+\|\xi\|^2)^{1-\eta}}
	<\infty, \quad  \text{for some $\eta\in(0\,,1)$}.
\end{equation}
The work of Sanz-Sol\'e and Sarra \cite{SanzSoleSarra2000,SanzSoleSarra2002} 
implies the sufficiency of this condition
for the local H\"older continuity of the solution to the heat equation with nonlinear noise and nonlinear drift terms. 
In that case, Theorem \ref{th:heat} completes the circle  by providing the necessity
of  \eqref{cond:H} in the case that $g$ is constant and the noise term is additive. 
Concerning the wave equation \eqref{SWE} in spatial dimension $d\in\{1,2,3\}$ (with $\rL=\Delta$), we see from \cite[Section 4]{SanzSoleSarra2000}  that condition \eqref{cc}
implies the H\"older continuity of the sample paths of the random-field solution. 
When $\Gamma$ is a Riesz kernel, the optimality of the H\"older index is proven in \cite[Chapter 5]{DalangSanzSole-2}. Thanks to the results of Sections \ref{sec:conv} and \ref{s5} we see that, in the case of a constant function $g$, and for any dimension $d\ge 1$, Condition \eqref{cc} is necessary and sufficient for the local H\"older continuity of the sample paths of the solution.                                                                                     

Next we say a few things about notation.
Throughout, we write $\Vert x\Vert$ for the Euclidean norm of any 
$x\in \R^d$, and $\|V\|_k = \{\E (|V|^k)\}^{1/k}$
for every $k\ge1$ and all random variables $V$. Moreover, we write
$f\lesssim g$ when the functions $f=f(x)$ and $g=g(x)$ are nonnegative
and there exists a number $c>0$ such that $f\le cg$ pointwise. And naturally, $g\gtrsim f$
is another way to say $f\lesssim g$. If $f\lesssim g$ and $g\lesssim f$, we write $f\asymp g$. In similar fashion, 
we might say that $f\lesssim g$ uniformly for all $x$ in some set $S$
when the respective restrictions $f_S$ and $g_S$ of $f$ and $g$
to $S$ satisfy the uniformly bound $f_S\le cg_S$ for some $c>0$. All vectors in $\R^d$ are consistently treated as
column vectors.
As usual, $\1_F$ denotes the indicator function of $F\subset\R^d$.

We conclude the Introduction with a few remarks about our methods and the existing 
literature. As has been observed on multiple occasions in the literature, 
the only general known method for establishing H\"older regularity of a stochastic process
is to appeal to a suitable form of the Kolmogorov continuity theorem,
which has to do with the moments of the increments of that process. 
In fact, there are beautiful formulations of this assertion that are rigorous theorems in their
own right; see Hahn and Klass \cite{HahnKlass}, Ibragimov \cite{Ibragimov79}, and Kon\^o \cite{Kono}.
Therefore, it does not come as surprise to know that the sufficiency half of our work also uses Kolmogorov's
continuity theorem. Our present challenge stems from the fact that 
a direct estimation/computation of the moments of the increments of the processes here do not naturally
lend themselves to a form that can be used in conjunction with the Kolmogorov continuity theorem.
In the case that $\rL$ is the Laplace operator, 
Sanz-Sol\'e and Sarr\`a \cite{SanzSoleSarra2000,SanzSoleSarra2002} overcome this challenge 
by appealing to H\"older's inequality, after they apply a stochastic version of the factorization method in 
semigroup theory; see also Dalang and Sanz-Sol\'e  \cite{DalangSanzSole} and
Li \cite{Li}. Here, we use a simple but powerful idea from fractional calculus (Lemma
\ref{lem:index}) that does not rely much on the particular form of $\rL$
and yields optimal results. There is a large literature that considers H\"older regularity of stochastic
PDEs; see for example Balan and Chen \cite{BalanChen}, Balan, Jolis, and Quer-Sardanyons \cite{BJQ}, 
Balan, Quer-Sardanyons, and Song \cite{BalanQuerSong},  
Chen and Dalang \cite{ChenDalang2014}, 
Dalang and Sanz-Sol\'e \cite{DalangSanzSole, DalangSanzSole-2},
Hu and Le \cite{HuLe},
Li \cite{Li}, and
Liu and Du \cite{LiuDu}.
There is also a large literature in which H\"older regularity is a  key step
in further analysis of the solution. For a small representative sampler of that
literature, see Bezdek \cite{Bezdek}, 
Boulanba, Eddhabi, and Mellouk \cite{BoulanbaEddhabiMellouk}, 
Chen and Kim \cite{ChenKim,ChenKim2002},
Dalang, Khoshnevisan, and Nualart \cite{DKN}, 
Dalang and Sanz-Sol\'e \cite{DalangSanzSole-3},
Dalang and Pu \cite{DalangPu},
Faugeras and Inglis \cite{FI},
Hu, Nualart, and Song \cite{HuNualartSong},
Hu, Huang, Nualart, and Sun \cite{HuHuangNualartSun},
Huang, Nualart, Viitasaari, and Zhang \cite{HuangNualartViitasaariZhang},
Lun and Warren \cite{LunWarren},
Misiats, Stanzhytskyi, and Yip \cite{MSY},
Nualart \cite{Nualart2013},
Rippl and Sturm \cite{RipplSturm},
and Sanz-Sol\'e and S\"u\ss\ \cite{SanzSoleSuss}.
Our methods can yield a different understanding of some
of these undertakings, and offer potential for extensions.

\section{L\'evy processes}\label{sec:Levy}
\label{s2}

We begin our discussion by making a few brief remarks and observations
about L\'evy processes. More details, and  further
information, on the general theory of L\'evy processes can be found in the 
monographs of Bertoin \cite{Bertoin}, Kyprianou \cite{Kyprianou}, and Sato \cite{Sato}.
Jacob \cite{Jacob} and Schilling \cite{Schilling} include the theory of L\'evy 
(and L\'evy like) processes from the point of view of harmonic analysis, 
particularly useful for the discussions that follow.

\subsection{The L\'evy-Khintchine formula}
Let $X=\{X(t)\}_{t\ge0}$ be a  L\'evy process on $\R^d$ that starts at the origin.
Recall that this means that $X(0)=0$, $t\mapsto X(t)$ is c\`adl\`ag, and 
$X(t+s)-X(s)$ is independent of $\{X(r);\, 0\le r\le s\}$ and distributed as $X(t)$
for all $s,t\ge0$. 
According to the L\'evy-Khintchine formula \cite{Bertoin,Jacob,Kyprianou,Sato,Schilling}, 
we can characterize the law of the
entire  $X$ via the distribution of $X(1)$ using the formula,
\begin{equation}\label{LKF1}
	\E\e^{i\xi\cdot X(t)} = \e^{-t\psi(\xi)}\qquad [t\ge0,\,z\in\R^d].
\end{equation}
In the preceding, $\psi:\R^d\to\C$ is called the \emph{characteristic exponent of $X$}, and has the form,
\begin{equation}\label{LKF}
	\psi(\xi) = -i\ba\cdot\xi
	+ \tfrac12\xi\cdot\bA\xi + \int_{\R^d} \left[ 1 - \e^{iy\cdot\xi}
	+ i(y\cdot\xi)\1_{B(0,1)}(y)\right]\nu(\d y),
\end{equation}
where $B(0\,,1)=\{y\in\R^d:\, \|y\|\le1\}$, $\ba\in\R^d$ is a constant,
\begin{equation}\label{sigma}
	\bA = \bQ'\bQ,
\end{equation}
for a $d\times d$ matrix $\bQ$, and $\nu$ is the L\'evy measure of $X$; that is,
$\nu$ is a $\sigma$-finite Borel measure $\R^d$
that satisfies $\nu\{0\}=0$ and
$\int_{\R^d}(1\wedge \|x\|)^2\,\nu(\d x)<\infty$. 
It is apparent from \eqref{LKF} that $\psi$ is continuous, and $\Re\psi$ is a real-valued,
in fact non negative, function. We tacitly use these facts throughout.

It is well known that
$X$ is a Feller process whose generator $\rL$ is a pseudo-differential operator
with constant symbols and Fourier multiplier $\hat{\rL}=-\psi$; that is,
$\int_{\R^d}(\rL f)(x)g(x)\,\d x = -(2\pi)^{-d}
\int_{\R^d} \hat{f}(\xi)\,\overline{\hat{g}(\xi)}\,\psi(\xi)\,\d\xi$
for every $f,g\in\cS(\R^d)$.
In particular,
test functions of rapid decrease form a core for the domain of the definition of
the generator $\rL$ of $X$, and the action of $\rL$ on $f\in\cS(\R^d)$ is
described as follows:
\begin{equation}\label{gen}
	\rL f = \ba\cdot\nabla f + \tfrac12 \text{Tr}\left(\bQ'\nabla^2 f \bQ\right)
	+ \int_{\R^d} \left[  f(\cdot+y)- f(y)-
	(y\cdot\nabla f)\1_{B(0,1)}(y)\right]\nu(\d y),
\end{equation}
where $\nabla^2=\nabla\nabla'$ denotes the Hessian matrix and $\bQ$ was defined in \eqref{sigma}.

Formulas \eqref{LKF} and \eqref{gen} are analytic ways of saying that we
can decompose $X$ as 
\begin{equation}\label{LKF:X}
	X(t) = \ba t + \bQ B(t) + Y(t)\qquad[t\ge0],
\end{equation}
where $B$ is a $d$-dimensional, standard
Brownian motion and $Y$ is an independent, pure-jump
L\'evy process. Note that the so-called $\bA$-Brownian motion
$\bQ B$ is a mean-zero Gaussian process
with $\Cov[\bQ B(s)\,,\bQ B(t)] = \min(s\,,t)\bA$
for all $s,t\ge0$, whence comes the name. The vector $\ba$ is sometimes referred to as the drift.
Finally, we point out that the structure theory of L\'evy process implies that --
and hinges on -- the fact that $\nu(A)$ is precisely equal
to the expected number of jumps
of $X$ (equiv.\ $Y$) in a Borel set $A\subset\R^d$.

In the Introduction, we mentioned that a standing assumption of this paper
is that the underlying L\'evy process is ``genuinely $d$-dimensional.'' By this we mean that
throughout we assume the following:
\begin{equation}\label{NoZeros}
	\psi^{-1}\{0\}=0,
\end{equation}
In order to see what this condition says, 
suppose to the contrary that there exists $\xi_0\neq0\in\R^d$ such that
$\psi(\xi_0)=0$.  According to the L\'evy-Khintchine formula \eqref{LKF1},
$0=\Re\psi(\xi_0) = \tfrac12 \|\bQ\xi_0\|^2 + 
\int_{\R^d}[ 1-\cos(\xi_0\cdot y)]\nu(\d y).$
Therefore, $\bQ\xi_0=0$ and $\cos(y\cdot\xi_0)=1$ for $\nu$-almost every $y\in\R^d$.
The former condition asserts that the rows of $\bQ$ are orthogonal to $\xi_0$, and the
latter condition is another way to say that $\nu$ is concentrated on the 0-dimensional set 
$\mathcal{Z}=\{y\in\R^d:\, y\cdot\xi_0\in(2\pi\Z)^d\}$. A final look at \eqref{LKF} now shows that
$\ba\cdot\xi_0=0$ also. Because the line of reasoning can be reversed, 
the description of this paragraph characterizes \eqref{NoZeros} and explains the reason for
using it to describe ``non degeneracy'': \eqref{NoZeros} says that either $\bQ'\nabla^2\bQ$ is
strongly elliptic, or $\nu$ does not concentrate on $\mathcal{Z}$, or both.
The class of all (possibly degenerate) L\'evy processes on $\R^d$ includes and extends the class of all
continuous-time random walks, also known as compound Poisson processes. 
Thus, the non concentration of $\nu$ on $\mathcal{Z}$ is an extension
of the notion of genuinely $d$-dimensional
continuous-time random walks, as can be found for  example in Spitzer \cite{Spitzer}. 
Condition \eqref{NoZeros} is natural enough that some authors assume it tacitly.

\subsection{Fractal indices and examples}

It is convenient to introduce two ``fractal indices'' 
that play a fundamental role in the formulation of the results of the paper. 
These indices show a trade-off between the singularity of the pseudo-differential 
operator $\mathcal{L}$ and the spectral spatial covariance of the noise $\dot F$. 

Our first regularity fractal index is
\begin{equation}\label{ind:H}
	\indH  = \sup\left\{a\in(0\,,1) :\
	\int_{\R^d} \frac{|\psi(\xi)|^a\,\mu(\d\xi)}{1+\Re\psi(\xi)}<\infty\right\},
\end{equation}
and our second fractal index is 
\begin{equation}\label{ind:L}
	\indL  = \sup\left\{ b\in(0\,,1):\, \int_{\R^d} \frac{\|\xi\|^{2b}\,\mu(\d\xi)}{1+\Re\psi(\xi)}
	<\infty\right\}.
\end{equation}
By convention, $\sup\varnothing=0$. Observe that \eqref{cond:H} holds iff $\indH >0$,
whereas $\indL >0$ iff \eqref{cond:R} holds. Furthermore,
\eqref{Bochner} readily yields 
\begin{equation}\label{indR<indH}
	\indL\le\indH.
\end{equation}

Let us briefly discuss some examples. 
The first was announced earlier in the Introduction.

\begin{proposition}\label{pr:ind:-}
	One can construct a pair $(\psi\,,\Gamma)$ that satisfies \eqref{Dalang},
	\eqref{NoZeros}, and
	$0=\indL<\indH=1$.
\end{proposition}

\begin{proof}
	We consider spatial dimension one, that is $d=1$, and  first construct $\Gamma$ via
	$\Gamma\{0\}=0$ and
	\[
		\frac{\d\Gamma(x)}{\d x}
		=\frac1\pi\int_0^\infty \frac{s\exp(-s)}{s^2+x^2}\,\d s
		\qquad\text{for all $x\in\R\setminus\{0\}$}.
	\]
	This defines a positive and positive definite and symmetric
	probability measure whose Fourier transform is given by
	$\d\mu(\xi)/\d \xi= (1+|\xi|)^{-1}$ for  $\xi\in\R$.
	We plan to construct a symmetric L\'evy process whose
	characteristic exponent satisfies \eqref{NoZeros}
	and $\psi(\xi)=\Re\psi(\xi)\asymp (\log \xi)^2$
	uniformly for all $\xi\ge\e^2$. Because $\psi$ is an even function,
	one can directly check, using \eqref{ind:H} and \eqref{ind:L},
	that this construction yields $0=\indL<\indH =1$ and completes the proof. 
	Therefore, it remains to construct such a characteristic exponent $\psi$.
	
	Consider the particular case of the L\'evy-Khintchine formula \eqref{LKF} with
	${\bf a}=0$, ${\bf A}=0$ and $\d\nu(x)/\d x = -|x|^{-1}\log |x|$ if $0< |x|<1/\e$
	and $\d\nu(x)/\d x=0$ otherwise;
	$\nu$ is a bona fide L\'evy measure since $\int_{\R\setminus\{0\}}
	(1\wedge x^2)\,\nu(\d x)<\infty$.  Also, 
	\[
		\psi(\xi) = 2\int_0^{1/\e}(1-\cos (|\xi|x)\log (1/x)\,\frac{\d x}{x}
		=2\int_0^{\xi/\e} (1-\cos y)\log\left( \frac{|\xi|}{y}\right) \,\frac{\d y}{y}
		\ge\log \xi\int_1^{\sqrt \xi} \left(\frac{1-\cos y}{y}\right) \,\d y,
	\]
	for $\xi\ge\e^2$. Moreover, the first identity holds 
	for all $\xi\in\R$. This readily shows that $\psi(\xi)>0$
	if $\xi\neq0$. Thus \eqref{NoZeros} follows. Also,
	since $\cos\le0$ on every interval $(\frac{n\pi}{2}\,,\frac{n\pi}{2}+\pi)$, 
	as $n$ ranges over $\N$, we may write
	\[
		\psi(\xi) \ge \log \xi\sum_{\substack{n\in\N:\\ (n\pi/2)+\pi \le \sqrt \xi}}
		\int_{n\pi/2}^{(n\pi/2)+\pi} \frac{\d y}{y }
		\ge\log \xi\sum_{\substack{n\in\N:\\ (n\pi/2)+\pi \le \sqrt \xi}} \frac{2}{n+2},
	\]
	uniformly for all $\xi\ge\e^2$. This immediately yields $\psi(\xi)\gtrsim(\log \xi)^2$
	uniformly for all $\xi\ge\e^2$. Conversely,
	\[
		\psi(\xi) \le 2 \int_0^1 \left(\frac{1-\cos y}{y}\right)\log(\xi/y) \,\d y
		+ 4\log \xi\int_1^{\xi/\e}\frac{\d y}{y}
		\lesssim\int_0^1 y\log(\xi/y)\,\d y + (\log \xi)^2 \lesssim (\log \xi)^2,
	\]
	valid uniformly for all $\xi\ge\e^2$. This completes the proof.
\end{proof}

These examples show, among
other things, that there is no general relationship between $\indL$ and $\indH$.

\begin{example}\label{ex:BM}
	First consider a L\'evy process which satisfies \eqref{LKF}
	for a non-singular matrix $\bA=\bQ'\bQ$; in terms of the representation
	\eqref{LKF:X}, this means that $X$ has a non-degenerate Brownian component. In that case,
	$\Re\psi(\xi) \ge\frac12\xi\cdot\bA\xi\gtrsim\|\xi\|^2$, uniformly for
	all $\xi\in\R^d$, and the nondegeneracy condition \eqref{NoZeros} holds. 
	In light of \eqref{Bochner}, we can see that Dalang's
	condition \eqref{Dalang} holds iff 
	$\int_{\R^d}(1+\|\xi\|)^{-2}\,\mu(\d\xi)<\infty.$ Moreover,
	\[
		\indL = \indH =1-\inf\left\{c\in(0\,,1):
		\,\int_{\R^d} \frac{\mu(\d\xi)}{1+\|\xi\|^{2c}}<\infty\right\}
		\qquad[\inf\varnothing=1],
	\]
	and \eqref{cond:R} and \eqref{cond:H} are equivalent to one another, as
	well as to the condition $0<\indL<1$, which is the condition of 
	Sanz-Sol\'e and Sarr\`a \cite{SanzSoleSarra2000,SanzSoleSarra2002}.
\end{example}

The above example covers all L\'evy processes that have a non-degenerate Gaussian component.
In particular the case
$\psi(\xi) = \tfrac{1}{2} \xi\cdot\bA\xi$, which corresponds to the Brownian motion.
In the remainder of our examples we consider L\'evy processes with a degenerate Gaussian component.
For simplicity of exposition, for the remainder of this section we consider in fact only $\bA=0$.

\begin{example}[Compound Poisson processes]\label{ex:CP}
	The process $X$ is compound Poisson (or a continuous-time random walk) iff
	$\nu$ is a finite measure. Let us recall that \eqref{NoZeros} holds
	iff $\nu$ does not concentrate on the countable set
	$\{y\in\R^d:\, y\cdot\xi_0\in(2\pi\Z)^d\}$ for
	any $\xi_0\in\R^d\setminus\{0\}$. This rules out examples such as the 
	standard 1-dimensional Poisson process.  According to \eqref{LKF},
	$\psi(\xi) =-i\tilde{\ba}\cdot\xi+\int_{\R^d}(1-\exp\{iy\cdot\xi\})\nu(\d y),$ where
	$\tilde{\ba} = \ba + \int_{B(0,1)}y\,\nu(\d y).$
	Since $0\le \Re\psi(\xi)\le 2\nu(\R^d)$, we see that condition \eqref{Dalang} holds
	iff the spectral measure $\mu$ is finite. This implies that
	$\Lambda(\d x) = f(x)\,\d x$, for a uniformly continuous, bounded, and
	positive-definite function $f:\R^d\to\R_+$
	\cite[p.\ 145]{don}. Suppose then that the preceding conditions are met.
	Then, 
	\[
		\indL = \sup\left\{ b\in(0\,,1):\, \int_{\R^d} \|\xi\|^{2b} \,\mu(\d\xi)<\infty\right\}
		\quad\text{and}\quad
		\indH = \sup\left\{ c\in(0\,,1):\, \int_{\R^d} |\tilde{\ba}\cdot\xi|^c \,\mu(\d\xi)<\infty\right\},
	\]
	where $\sup\varnothing=0$, as before.
	Because $ |\tilde{\ba}\cdot\xi| \le \|\tilde{\ba}\|\|\xi\|$,
	we have $\frac12\indL\le\indH$ and the inequality can be strict 
	because $\indH=1$ generically when $\tilde{\ba}=0$.
\end{example}

\begin{example}[Radially symmetric stable processes]\label{ex:SSalpha}
	Choose and fix some $\alpha\in(0\,,2)$ and let $\d\nu(x)/\d x=C\|x\|^{-d-\alpha}$,
	where $C>0$ is chosen to ensure that $\psi(\xi)=\|\xi\|^\alpha$ for all $\xi\in\R^d$.
	The resulting process $X$ is called the radially symmetric $\alpha$-stable process. Clearly,
	\eqref{NoZeros} holds, and the generator
	is given by the following principle-value integral,
	\[
		(\rL f)(x) = C\int_{\R^d}\frac{f(x+y)-f(y)}{\|y\|^{d+\alpha}}\,\d y
		\quad\text{for all $x\in\R^d$ and $f\in\cS(\R^d)$};
	\]
	see \eqref{gen}. That is, $\rL$ is [a multiple of] the fractional Laplace operator
	$-(-\Delta)^{\alpha/2}$. In this case, Dalang's condition holds iff 
	$\int_{\R^d}(1+\|\xi\|)^{-\alpha}\,\mu(\d x) <\infty$. And if this holds then
	\[
		\indH= \sup\left\{a\in(0\,,1):\,
		\int_{\R^d}\frac{\mu(\d\xi)}{1+\|\xi\|^{\alpha(1-a)}} <\infty\right\},\quad
		\indL= \sup\left\{b\in(0\,,2/\alpha):\,
		\int_{\R^d}\frac{\mu(\d\xi)}{1+\|\xi\|^{\alpha(1-b)}} <\infty\right\}.
	\]
	Because $\alpha<2$, two different regimes emerge:
	On one hand, if $\mu(\R^d)=\infty$, then $\indL=(\alpha/2)\indH$ and the latter is of course given
	by the previous display. [Send $\alpha\uparrow2$ in order to somewhat informally return to Example
	\ref{ex:BM}.]
	On the other hand, if $\mu(\R^d)<\infty$, then it can happen that 
	$\indL>(\alpha/2)\indH$; in fact, in that case, careful examination of the preceding display yields
	the following:
	\[
		\indH=1
		\quad\text{and}\quad
		\indL=\frac\alpha2 + \frac12\sup\left\{c\in(0\,,2-\alpha):\,
		\int_{\R^d}\|\xi\|^c\,\mu(\d\xi)  <\infty\right\}.
	\]
\end{example}

\begin{example}[Asymmetric Cauchy processes on $\R$]\label{ex:Cauchy}
	Set $d=1$ and $\ba=0$, and 
	select $\d\nu(x)/\d x = c_1x^{-2}\,\d x$ if $x>0$ and
	$\d\nu(x)/\d x=c_2x^{-2}\,\d x$ if $x<0$, where $c_1,c_2>0$ are fixed numbers.
	The resulting process is the so-called Cauchy process on $\R$, and is described by
	$\psi(\xi) = |\xi| + ih\xi\log|\xi|$ for $\xi\neq0$, and $\psi(0)=0$.
	The \emph{asymmetry parameter} $h$ depends on $c_1$ and $c_2$ and
	can be made to be any number in $[-2/\pi\,,2/\pi]$ by suitably adjusting $c_1,c_2$.
	The symmetric case $h=0$ is subsumed by Example \ref{ex:SSalpha} $[d=\alpha=1]$.
	We now look at the asymmetric case, $h\neq0$. Dalang's condition \eqref{Dalang}
	holds iff $\int_{\R}(1+|\xi|)^{-1}\,\mu(\d\xi)<\infty$, the same as it does in the 
	symmetric case. Moreover, in the asymmetric case, $|\psi(\xi)|\sim |h\xi\log|\xi||$ as
	$\xi\to\pm\infty$. Since $\log|\xi|=|\xi|^{o(1)}$ as $|\xi|\to\infty$, it follows that,
	regardless of the strength of the asymmetry
	parameter $h$,
	\[
		\indL={\tfrac12}\indH\qquad\text{and}\qquad
		\indH=\sup\left\{a\in(0\,,1):\,\int_{\R}\frac{\mu(\d\xi)}{1+|\xi|^{1-a}}<\infty\right\}.
	\]
\end{example}

\subsection{The Kolmogorov-Fokker-Planck equation}
The fundamental solution to the corresponding heat operator $\partial_t-\rL$ is
the solution to the integro-differential equation
$\partial_t \varphi = \rL\varphi$ on $(0\,,\infty)\times\R^d$,
subject to $\varphi(0)= \delta_0$ on $\R^d$.
Since $\hat{\rL}=-\psi$, we can apply the Fourier transform
in order to see that the fundamental solution to $\partial_t-\rL$ is
the following probability-measure-valued function of $t$:
\begin{equation}\label{p_t(F)}
	p(t\,,F) = \P\{X(t)\in F\}\qquad\text{for all $t\ge0$ and Borel sets $F\subset\R^d$}.
\end{equation}
These are the so-called \emph{transition functions} of the L\'evy process $X$.

The preceding discussion tells us that if $u_0$ is a tempered, nonnegative-definite
Borel measure on $\R^d$ then the associated Kolmogorov-Fokker-Planck equation,
$\partial_t f = \rL f$ on $(0\,,\infty)\times\R^d$,
subject to $f(0)=u_0$ on $\R^d$,
is well posed, and has a unique solution that is a function from $\R_+$ to the space of
tempered distributions (Borel probability measures, in fact),
and that unique solution is in fact described by 
$f(t\,,\cdot)=p(t)*u_0$ at time $t\ge0$. The following addresses
the local H\"older regularity of that solution
when $u_0$ is a random field satisfying some moment-type constraints. 

\begin{proposition}\label{pr:heat}
	Let $u_0=\{u_0(x);\, x\in\R^d\}$ be a random field. Assume that
	there exist real numbers $k\ge1$ and $\eta\in(0\,,1]$, and a constant $C:=C(k)>0$ such that
	\begin{equation}\label{u_0}
		\sup_{z\in\R^d}\E\left(|u_0(z)|^k\right)<\infty
		\quad\text{and}\quad
		\E\left( |u_0(x) - u_0(y)|^k\right) \le C\|x-y\|^{\eta k}\qquad\text{uniformly
		for all $x,y\in\R^d$}.
	\end{equation}
	Then, there exists $c=c(k)>0$ such that
	\begin{equation}\label{mom}
		\E\left( \left| (p(t+\varepsilon)*u_0)(-x) - (p(t)*u_0)(-y) \right|^k \right)
		\le c \left( \varepsilon^{1/2} + \|x-y\|\right)^{\eta k},
	\end{equation}
	uniformly for all $t\ge0$, $\varepsilon\in[0\,,1]$, and $x,y\in\R^d$.
\smallskip
	
	If \eqref{u_0} holds for every $k\ge1$, then for every $a\in(0\,,\eta)$,
$(t\,,x)\mapsto (p(t)*u_0)(-x)$ a.s.\ belongs to
	$C^{a/2,a}_{\text{\it loc}}(\R_+\times\R^d)$.
	 \end{proposition}

The proposition rests on the following result, which is a quantitative version of the assertion
that every L\'evy process on $\R^d$ is locally subdiffusive.

\begin{lemma}\label{lem:diffuse}
	$\E( \|X(t)\|^2\wedge 1) \lesssim t$ uniformly for all $t\in[0\,,1]$.
\end{lemma}

\begin{proof}
	Let $X_i$ denote the $i$th coordinate process of $X$ and observe that
	$\E(\|X(t)\|^2\wedge 1) \lesssim\sum_{i=1}^d
	\E( |X_i(t)|^2\wedge 1)$, uniformly for all $t\ge0$.
	Since every $X_i$ is a one-dimensional L\'evy process, the above reduces the problem
	to a one-dimensional problem. Therefore, we may and will assume without loss in generality
	that $d=1$. In this case, \eqref{LKF:X} reduces to the existence of real numbers $a,q\ge0$
	such that $X(t)=at+qB(t)+Y(t)$ for all $t\ge0$, where $B$ is a standard linear Brownian motion
	and $Y$ is an independent one-dimensional, pure-jump L\'evy process. It follows from this that
	\begin{equation}\label{EX}
		\E( |X(t)|^2\wedge 1) \lesssim t^2 + 
		\E(  |B(t)|^2 \wedge 1  ) 
		+ \E( |Y(t)|^2 \wedge 1 )
		\lesssim t + \E( |Y(t)|^2 \wedge 1),
	\end{equation}
	uniformly for all $t\in [0\,,1]$.
	Thanks to the structure theory of L\'evy processes, we may write
	$Y(t) = S(t) + L(t)$ for all $t\ge0$ where $S$ and $L$ are independent one-dimensional
	L\'evy processes, and 
	\[
		L(t) = \sum_{s\in[0,t]}(\Delta Y)(s)\1_{\{(\Delta Y)(s)>1\}}\qquad\text{for all $t\ge0$},
	\]
	where $(\Delta Y)(s) = Y(s)-Y(s-)$. The letters ``$L$'' and ``$S$'' respectively refer 
	to the ``large jumps'' and the ``small jumps'' of $Y$. The collection of all of the jumps of
	$Y$ that are $>1$ in magnitude has a Poisson distribution with parameter 
	$t\nu(\R\setminus[-1\,,1])$. Therefore, for every $t\ge0$,
	$\P\{ L(s)\neq0\text{ for some $s\in[0\,,t]$}\}
	= 1 - \exp\{-t\nu(\R\setminus[-1,1])\}\lesssim t$,
	and hence
	\begin{equation}\label{EY}
		\E( |Y(t)|^2 \wedge 1 )
		\le \E( |S(t)|^2 \wedge 1 ) + \P\{L(t)\neq0\}  \lesssim \E(|S(t)|^2) + t,
	\end{equation}
	uniformly for all $t\ge0$.
	The process $S$ is L\'evy with zero drift and zero diffusion, and its L\'evy measure is the restriction
	of $\nu$ to $[-1\,,1]$. Therefore,  the structure theory of L\'evy processes
	implies that the stochastic process
	$\{|S(t)|^2-t\int_{-1}^1 y^2\,\nu(\d y);\,t\ge0\}$ is a mean-zero martingale, whence
	$\E(|S(t)|^2)= t\int_{-1}^1 y^2\,\nu(\d y)$ for every $t\ge0$.
	Combine with \eqref{EY} and \eqref{EX} to complete the proof.
\end{proof}

\begin{proof}[Proof of Proposition \ref{pr:heat}]
	The moment bound \eqref{mom} and Kolmogorov's continuity theorem together imply the H\"older
	regularity statement. Therefore, it suffices to establish \eqref{mom}.
	
	Without any loss in generality, we can and will assume that $u_0$
	and $X$ are independent. With this in mind,
	let us note that $(p(t)*u_0)(-x) = \E[u_0(x+X(t))]$ for all $t\ge0$ and $x\in\R^d$.
	Therefore, if we write $\mathbb{E}(\,\cdots) = \E(\,\cdots\mid u_0)$, and denote by 
	$L^k(u_0)$ the space $L^k(\Omega)$ endowed with a regular version of the 
	conditional probability $\P(\,\cdots\mid u_0)$,
	we can deduce from disintegration, \eqref{u_0}, and the triangle inequality that
	\begin{equation}\label{C1}\begin{split}
		\left\| \left( p(t)*u_0\right)(-x) - \left( p(t) *u_0\right)(-z) \right\|_{L^k(u_0)}
			&\le \mathbb{E}\left( \left\| u_0(x+X(t)) - u_0(z+X(t))\right\|_{L^k(u_0)}\right)\\
		&\le C^{1/k}\|x-z\|^\eta,
	\end{split}\end{equation}%
	valid uniformly for all $k\ge1$, $x,z\in\R^d$, and $t\ge0$. On the other hand,
	set $D=\sup_{z\in\R^d}\|u_0(z)\|_k$ to infer from \eqref{u_0} that,
	for all $k\ge1$, the following holds uniformly for all $x\in\R^d$, and $t,\varepsilon\ge0$:
	\begin{align*}
		&\left\| \left( p(t+\varepsilon)*u_0\right)(-x) - \left( p(t) *u_0\right)(-x) \right\|_{L^k(u_0)}
			\le \mathbb{E}\left( \left\| u_0(x+X(t+\varepsilon)) - 
			u_0(x+X(t))\right\|_{L^k(u_0)}\right)\\
		&\le \E\left[ C^{1/k}\|X(t+\varepsilon) - X(t)\|^\eta \wedge 2D\right] = 
			\E\left[ C^{1/k}\|X(\varepsilon)\|^\eta \wedge 2D\right]
			\lesssim \E\left( \|X(\varepsilon)\|^\eta \wedge 1\right)\le
			\left\{\E\left( \|X(\varepsilon)\|^2 \wedge 1\right)\right\}^{\eta/2},
	\end{align*}
	thanks to Jensen's inequality. 
	   Lemma \ref{lem:diffuse} then
	implies that for every $k\ge1$,
	\[
		\E\left(\left| \left( p(t+\varepsilon)*u_0\right)(-x) - \left( p(t) *u_0\right)(-x) 
		\right|^k\right)
		\lesssim \varepsilon^{\eta k/2},
	\]
	uniformly for all $x\in\R^d$, and $t,\varepsilon\ge0$. Combine this with
	\eqref{C1} to conclude \eqref{mom} whence the lemma.
\end{proof}

\section{The stochastic convolution}
\label{sec:conv}
In this section we study the sample-function regularity of the Fourier-Laplace type stochastic 
convolution,
\begin{equation}\label{inf:cov}
	\int_{(0,t)\times\R^d}\varphi(t-s\,,x-y)\dot{\rf}(s\,,y)\,\d s\,\d y
	:= \int_{(0,t)\times\R^d}\varphi(t-s\,,x-y)\,\rf(\d s,\d y) :=
	(\varphi*\rf)(t\,,x),
\end{equation}
of a  deterministic function $\varphi:\R_+\times\R^d\to\R$ and the noise $\dot{\rf}$, using 
ideas of Wiener on stochastic integration. 

Fix $\varphi:\R_+\times\R^d\to\R$ deterministic, and
 $t_1,t_2\ge0$, $x_1,x_2\in\R^d$. If we formally,
and somewhat liberally, exchange expectations
with integrals, and if we treat $\delta_0$ and $\Gamma$ in \eqref{Cov} as nice functions,
then we see from \eqref{Cov} and \eqref{inf:cov} that
\begin{align*}
	&\Cov\left[ (\varphi*\rf)(t_1,x_1) ~,~ (\varphi*\rf)(t_2\,,x_2)\right]\\
	&= \int_{(0,t_1)\times(0,t_2)}\d s_1\,\d s_2\int_{\R^d\times\R^d}\d y_1\,\d y_2\
		\varphi(t_1-s_1\,,x_1-y_1)\varphi(t_2-s_2\,,x_2-y_2)
		\E\left[ \dot{\rf}(s_1\,,y_1)\dot{\rf}(s_2\,,y_2)\right]\\
	&=\int_0^{t_1}\d s_1\int_0^{t_2}\d s_2\int_{\R^d\times\R^d}\d y_1\,\d y_2\
		\varphi(t_1-s_1\,,x_1-y_1)\varphi(t_2-s_2\,,x_2-y_2)
		\delta_0(s_1-s_2)\Gamma(y_1-y_2)\\
	&=\int_0^{t_1\wedge t_2}
		\left( \varphi(t_1-s)*\varphi(t_2-s)*\Gamma\right)(x_1-x_2)\,\d s.
\end{align*}
The first two identites are not rigorous and contain ill-defined integrals. However,
the final expression
makes perfectly good sense as a Lebesgue integral
when $\varphi$ is a nice space-time function. Under appropriate conditions on  $\varphi$ one can in fact apply inverse Fourier transforms to see that
\begin{equation}\label{phi*f:Cov}
	\Cov\left[ (\varphi*\rf)(t_1,x_1) ~,~ (\varphi*\rf)(t_2\,,x_2)\right]
	= (2\pi)^{-d} \int_0^{t_1\wedge t_2}\d s\int_{\R^d}\mu(\d\xi)\
	\hat{\varphi}(t_1-s\,,\xi)\overline{\hat{\varphi}(t_2-s\,,\xi)} \,\e^{i(x_1-x_2)\cdot\xi},
\end{equation}
where $\hat{\varphi}(t\,,\xi) = \int_{\R^d}\exp(i\xi\cdot x)\varphi(x)\,\d x,$
whenever the preceding integral makes sense as a Lebesgue integral, and we recall that
$\mu=\hat{\Gamma}$ is a tempered Borel measure on $\R^d$ since $\Gamma$ is a
tempered, nonnegative-definite Borel measure on $\R^d$.

In order to rigorously construct the stochastic convolution, one merely reverse-engineers
the preceding ``calculation'' and starts with \eqref{phi*f:Cov} as a definition.
We outline the details next.

\begin{definition}
	Let $\cK$ denote the class of all functions 
	$\varphi:\R_+\to\cS'(\R^d)$ such that:
	\begin{compactenum}[(1)]
		\item $\hat{\varphi}(t)\in L^2(\mu)$ for every $t\ge0$; and
		\item There is a version $\hat{\varphi}:\R_+\times\R^d\to\mathbb{C}$ that
			is jointly measurable that satisfies
			\begin{equation}\label{D}
				0< \int_0^t\d s\int_{\R^d}\mu(\d\xi)\ |\hat{\varphi}(s\,,\xi)|^2<\infty
				\qquad\text{for every $t>0$}.
			\end{equation}
	\end{compactenum}
	We refer to elements of $\cK$ as \emph{kernels}.
\end{definition}
 The fundamental solutions to some PDEs provide typical examples of \emph{kernels}.
The strict positivity condition in \eqref{D} is there only to avoid degenerate cases.

Now, choose and fix a kernel $\varphi\in\cK$, and
denote the quantity on the right-hand side of \eqref{phi*f:Cov} by $\< (t_1\,,x_1)\,,(t_2\,,x_2)\>$.
We may view the latter as a pairing
between $(t_1\,,x_1)$ and $(t_2\,,x_2)$ for all $(t_1\,,x_1),(t_2\,,x_2)\in\R_+\times\R^d$.
One can readily see that this pairing defines
a nonnegative-definite Hermitian product on $(\R_+\times\R^d)\times
(\R_+\times\R^d)$. The general theory of Gaussian processes then ensures that 
there exists a $(d+1)$-parameter Gaussian process $\varphi*\rf=\{(\varphi*\rf)(t\,,x);\, t\ge0,x\in\R^d\}$
whose covariance is given by the above pairing; that is, 
$\varphi*\rf$ satisfies \eqref{phi*f:Cov}. Observe that $\varphi*\rf$ is stationary in space. 
One can in fact construct $\varphi*\rf$ for a slightly larger
family of non-random integrands by replacing Item (2) above by the condition that
the function $t\mapsto \|\hat{\varphi}(t)\|_{L^2(\mu)}$ is in $L^2_{\text{\it loc}}(\R_+)$. But the
presently assumed joint measurability of $\hat\varphi$ is convenient.

The goal of this section is to find close-to-optimal local H\"older regularity results for
stochastic convolutions of the above type. 
For the  remainder of this section
we choose and fix a nonrandom number $T>0$ and a 
kernel $\varphi\in\cK$, and define four
indices that play a key role in the regularity of the sample functions of $\varphi*\rf$.
In the particular examples of the stochastic heat and wave equations, we will prove in Lemma \ref{lem:ind:H} and Lemma \ref{lem:ind:W}, respectively, their relationship with the indices $\indH$ and $\indL$ (defined in \eqref{ind:H}
and \eqref{ind:L}, respectively).

The first index is defined as follows:
\begin{equation}\label{phi*f:indR}
	\cI_R(T) = \sup\left\{\eta\in(0\,,1):\,
	\int_0^T\d s\int_{\R^d}\mu(\d\xi)\ \|\xi\|^{\eta}
	\| \hat{\varphi}(s\,,\xi)|^2 <\infty\right\},
\end{equation}
where $\sup\varnothing=0$. 
\medskip

\noindent{\em H\"older regularity in space}
\begin{proposition}\label{pr:phi*f:x}
	If $\cI_R(T)>0$ then, for all $a\in(0\,,\cI_R(T))$
	\[
		\left\| (\varphi*\rf)(t\,,h) - (\varphi*\rf)(t\,,0)\right\|_2 \lesssim
		\|h\|^{a/2},
	\]
	uniformly for all $t\in[0\,,T]$ and $h\in\R^d$ that satisfy $\|h\|\le1$.
	
	On the other hand, if $\cI_R(T)=0$, then
	\[
		\limsup_{h\to0} \|h\|^{-a}
		\left\| (\varphi*\rf)(t\,,h) - (\varphi*\rf)(t\,,0)\right\|_2 =\infty
		\qquad\text{for every $a>0$ and $t\ge T$}.
	\]
	
	Furthermore, for any $a\in(\cI_R(T),\infty)$, 
	\[
	\limsup_{h\to 0}  \|h\|^{-a}
		\left\| (\varphi*\rf)(t\,,h) - (\varphi*\rf)(t\,,0)\right\|_2 =\infty
		\qquad\text{for every $t\ge T$}.
		\]
\end{proposition}

\begin{proof}
	First, let us suppose that $\cI_R(T)>0$ for a fixed $T>0$.
	Choose and fix some $a\in(0\,, \cI_R(T))$.
	The definitions of $\cK$ and $\cI_R(T)$ together imply that
	\begin{equation}\label{A}
		\int_0^T\d s\int_{\R^d}  \mu(\d\xi) \left(1+\|\xi\|^{a}\right) \
		\left| \hat{\varphi}(s\,,\xi)\right|^2 <\infty.
	\end{equation}
	
	According to \eqref{phi*f:Cov},  for any $h\in\R^d$,
	\begin{equation}\label{REF}\begin{split}
		\E\left( \left| (\varphi*\rf)(t\,,h) - (\varphi*\rf)(t\,,0)\right|^2\right)
			&= \frac{2}{(2\pi)^d}\int_0^t\d s\int_{\R^d}\mu(\d\xi)\
			\left| \hat{\varphi}(s\,,\xi)\right|^2 \left[ 1-\cos(h\cdot\xi)\right]\\
		&\le\frac{2}{(2\pi)^d}\int_0^T\d s\int_{\R^d}\mu(\d\xi)\
			\left| \hat{\varphi}(s\,,\xi)\right|^2  \left(\|\xi\|^2 \|h\|^2\wedge 1\right),
	\end{split}\end{equation}
	uniformly for all $t\in[0\,,T]$. Define
	\[
		g(\varepsilon) = \sup_{t\in(0,T]}\sup_{\substack{h\in\R^d:\\\|h\|\le\varepsilon}}
		\E\left( \left| (\varphi*\rf)(t\,,h) - (\varphi*\rf)(t\,,0)\right|^2\right)
		\qquad\text{for all $\varepsilon>0$}.
	\]
	As part of the definition of $\cI_R(T)$,
	we are assured that $a\in(0\,,1)$; see \eqref{phi*f:indR}.
	Therefore, \eqref{REF} and Lemma \ref{lem:A} of the appendix (with $b=2$ and $c:=2a<2$)
	together show that
	$\int_0^1\varepsilon^{-1-2a}g(\varepsilon)\,\d\varepsilon<\infty.$
	Since $g$ is non decreasing and measurable, Lemma	
	\ref{lem:index} of the appendix ensures that 
	$g(\varepsilon)\lesssim\varepsilon^{2a}$ uniformly for all $\varepsilon\in(0\,,1)$.
	The first assertion of the proposition follows. 
	
	For the second part, we suppose $\cI_R(T)=0$ for a fixed $T>0$.
	Since $\varphi\in\cK$, \eqref{phi*f:indR} assures us that
	\[
		\int_0^t\d s\int_{\R^d}\mu(\d\xi)\
		|\hat{\varphi}(s\,,\xi)|^2\|\xi\|^{a}=\infty
		\qquad\text{for all $t\ge T$ and $a\in(0\,,1)$}.
	\]
	Choose and fix $a\in(0\,,1)$ and $t\ge T$. In the following we will use the fact that
	\begin{equation}\label{cos}
		\int_{\R^d} \left( \frac{1-\cos(\xi\cdot h)}{\|h\|^{d+a}}\right)\d h
		= c \|\xi\|^{a}\qquad\text{for all $\xi\in\R^d$},
	\end{equation}
	where $c=c(d\,,a)>0$ is a real number.
	This holds because the left-hand side of \eqref{cos} converges whenever $a\in(0\,,1)$
	and defines a radial function of $\xi$.
	
	Define
	\[
		f(h) = \E\left( \left| (\varphi*\rf)(t\,,h) - (\varphi*\rf)(t\,,0)\right|^2\right)\qquad\text{for
		every $h\in\R^d$}.
	\]
	We may combine the first line of \eqref{REF} with \eqref{cos} to see that, 
	for all $\xi\in\R^d$ and $t\ge T$,
	\begin{equation}\label{f:infty}
		\int_{\R^d} f(h)\frac{\d h}{\|h\|^{d+a}}
		= \frac{2c}{(2\pi)^d}\int_0^t\d s\int_{\R^d}\mu(\d\xi)\
		\left| \hat{\varphi}(s\,,\xi)\right|^2 \|\xi\|^{a}=\infty
	\end{equation}
	Define $f_n = \sup_{\|x\|\le\exp(-n)}f(x)$ for every $n\in\Z_+$
	and observe that
	\[
		\sum_{n=0}^\infty \e^{na} f_{n}
		\gtrsim \sum_{n=1}^\infty f_{n-1}
		\int_{\e^{-n}\le\|h\|\le \e^{-n+1}} \frac{\d h}{\|h\|^{d+a}}\ge
		\int_{\|h\|\le1/\e} f(h)\,\frac{\d h}{\|h\|^{d+a}}.
	\]
	Because of \eqref{REF}, $f$ is bounded uniformly from above and therefore,
	$\int_{\|h\|>1/\e} f(h)\|h\|^{-d-a}\, \d h$
        converges. Hence by 
	\eqref{f:infty} $\int_{\|h\|\le1/\e} f(h)\|h\|^{-d-a}\, \d h=\infty$. Consequently,
	 $\sum_{n=0}^\infty \e^{na}f_{n}=\infty.$
	This proves that $\ell=\sup_{n\ge 1}\e^{nb}f_{n} =\infty$
	whenever $b>a$ for 
	$\sum_{n=0}^\infty\e^{na} f_n
	\le \ell\sum_{n=0}^\infty \e^{-n(b-a)}<\infty$
	otherwise. Because $\ell=\infty$, a monotonicity argument implies that
	$\limsup_{h\to0} \|h\|^{-b} f(h) =\infty$
	for every $b>a$.
	Since $a>0$ is arbitrary, this completes the proof of the proposition.
	
	For the proof of the last statement, we observe that, if $a\in(\cI_R(T),\infty)$ then 
	\[
	\int_0^T \d s\int_{\R^d}\mu(\d\xi)|\hat{\varphi}(s\,,\xi)|^2\|\xi\|^{a}=\infty
		\]
Defining $f_T(h) = \E\left( \left| (\varphi*\rf)(T\,,h) - (\varphi*\rf)(T\,,0)\right|^2\right)$, with the
same arguments used to obtain \eqref{f:infty}, we have 
$\int_{\R^d} f_T(h)\,\d h/\Vert h\Vert^{d+a}=\infty$. This yields the assertion.
	\end{proof}
	
	In the particular instances of the heat and wave equation, it turns out that $\mathcal{I}_R(T)$ is independent of $T$: see Lemmas \ref{lem:ind:H} and \ref{lem:ind:W}, respectively.

\medskip

\noindent{\em H\"older regularity in time}
\smallskip

We introduce the remaining relevant indices of this section:
\begin{equation}\label{phi*f:indHSS'}\begin{split}
	\cI_H &= \sup\left\{ b\in(0\,,1):\,
		\int_0^1\frac{\d s}{s^b}\int_{\R^d}\mu(\d\xi)\ 
		| \hat{\varphi}(s\,,\xi)|^2 <\infty\right\},\\
	\overline{\cI}_H(T)  &= \sup\left\{b\in(0\,,1):\,
		\int_0^1\frac{\d\varepsilon}{\varepsilon^{1+b}}
		\int_0^T \d s\int_{\R^d}\mu(\d\xi)\
		\left| \hat{\varphi}(s + \varepsilon \,,\xi)-\hat{\varphi}(s\,,\xi)\right|^2
		<\infty\right\},\\
	\underline{\cI}_H(T)  &= \sup\left\{b\in(0\,,1):\,
		\int_0^1\frac{\d\varepsilon}{\varepsilon^{1+b}}\sup_{r\in[0,\varepsilon]}
		\int_0^T\d s\int_{\R^d}\mu(\d\xi)\
		\left| \hat{\varphi}(s + r \,,\xi)-\hat{\varphi}(s\,,\xi)\right|^2
		<\infty\right\},
\end{split}\end{equation}
where $T>0$. Clearly, $0\le\underline{\cI}_H(T)\le\overline{\cI}_H(T)\le1$.

\begin{proposition}\label{pr:phi*f:t}
	Assume  $\cI_H\wedge\underline{\cI}_H(T)>0$. Then, for any  $b\in(0\,,\cI_H\wedge\underline{\cI}_H(T))$,
	
	\[
		 \left\|(\varphi*\rf)(t+\varepsilon\,,0) - (\varphi*\rf)(t\,,0) \right\|_2
		 \lesssim\varepsilon^{b/2},
	\]
	uniformly for all $t\in[0\,,T]$ and $\varepsilon\in(0\,,1)$. 
	
	On the other hand, for all $t>T$ and $b > \cI_H\wedge\overline{\cI}_H(T)$,
	\[
		\limsup_{\varepsilon\to0^+}\varepsilon^{-b/2}
		\left\|(\varphi*\rf)(t+\varepsilon\,,0) - (\varphi*\rf)(t\,,0) \right\|_2 =\infty.
		\]
\end{proposition}

\begin{proof}
	Choose and fix $t\in[0\,,T]$.
	Owing to \eqref{phi*f:Cov}, we may write for every $\varepsilon>0$,
	\[
		\E\left(\left|(\varphi*\rf)(t+\varepsilon\,,0) - (\varphi*\rf)(t\,,0) \right|^2\right)
		= \frac{I_1(\varepsilon) + I_2(\varepsilon) }{(2\pi)^d},
	\]
	where
	\begin{align*}
		I_1(\varepsilon) &= \int_t^{t+\varepsilon}\d s\int_{\R^d}\mu(\d\xi)\
			\left| \hat{\varphi}(t + \varepsilon -s\,,\xi)\right|^2
			= \int_0^\varepsilon\d s\int_{\R^d}\mu(\d\xi)\
			\left| \hat{\varphi}(s \,,\xi)\right|^2,\\
		I_2(\varepsilon,t) &= \int_0^t\d s\int_{\R^d}\mu(\d\xi)\
			\left| \hat{\varphi}(s + \varepsilon \,,\xi)-\hat{\varphi}(s\,,\xi)\right|^2.
	\end{align*}
	Tonelli's theorem implies that, for every $a>0$,
	\[
		\int_0^1\frac{\d\varepsilon}{\varepsilon^{1+a}}\ I_1(\varepsilon)
		= \int_0^1\d s\int_s^1\frac{\d\varepsilon}{\varepsilon^{1+a}}
		\int_{\R^d}\mu(\d\xi)\
		\left| \hat{\varphi}(s\,,\xi)\right|^2
		=\frac{1}{a}\int_0^1\left(\frac{1}{s^a}-1\right)\d s
		\int_{\R^d}\mu(\d\xi)\
		\left| \hat{\varphi}(s\,,\xi)\right|^2.
	\]
	Assume that $a\in(0\,,\cI_H)$. Because $\varphi\in\cK$, it follows that
	\[
		\int_0^1\frac{\d\varepsilon}{\varepsilon^{1+a}}\ I_1(\varepsilon)
		<\infty.
	\]
	Since $I_1$ is monotone, we can apply the integral test from calculus 
	to obtain $\sum_{n=0}^\infty \e^{a n}I_1(\e^{-n})<\infty$.
	Consequently, if $a\in(0\,,\cI_H)$ then $I_1(\e^{-n})=o(\e^{a n})$ as $n\to\infty$ along integers;
	see also Lemma \ref{lem:index} of the appendix.
	We deduce from this the following:
	\begin{equation}\label{I_1:1}
		\text{For every }a\in(0\,,\cI_H)\text{ there exists a number $C>0$
		such that }
		I_1(\varepsilon) \le C\varepsilon^a\text{ for all $\varepsilon\in(0\,,1)$}.
	\end{equation}
	We also learn from Lemma \ref{lem:index} that for every $a>\cI_H$,

	\begin{equation}\label{I_1:2}
		\limsup_{\varepsilon\to0^+}\varepsilon^{-a}I_1(\varepsilon)
		=\infty.
	\end{equation}
	
	For all $a>0$, define
	\begin{align*}
		\mathscr{A}(a) &=\int_0^1\frac{\d\varepsilon}{\varepsilon^{1+a}}\ I_2(\varepsilon)
			=\int_0^1\frac{\d\varepsilon}{\varepsilon^{1+a}}
			\int_0^T\d s\int_{\R^d}\mu(\d\xi)\
			\left| \hat{\varphi}(s + \varepsilon \,,\xi)-\hat{\varphi}(s\,,\xi)\right|^2,\\
		\mathscr{B}(a)& =\int_0^1\frac{\d\varepsilon}{\varepsilon^{1+a}}\ \sup_{r\in[0,\varepsilon]} I_2(r)
			=\int_0^1\frac{\d\varepsilon}{\varepsilon^{1+a}}
			\sup_{r\in[0,\varepsilon]}\int_0^T\d s\int_{\R^d}\mu(\d\xi)\
			\left| \hat{\varphi}(s + r \,,\xi)-\hat{\varphi}(s\,,\xi)\right|^2.
	\end{align*}
	The justifications of \eqref{I_1:1} and \eqref{I_1:2} can now be repurposed in
	order to show the following:
	\begin{compactenum}
	\item If $a\in(0\,,\underline{\cI}_H(T))$, then $ \mathscr{B}(a)<\infty$ and therefore
		$I_2(\varepsilon,t)\lesssim\varepsilon^b$ uniformly for all $t\in[0,T]$, $\varepsilon\in(0\,,1)$
		and $b\in(0\,,a)$.
		Along with \eqref{I_1:1}, this yields the first statement of the proposition;
	\item $\mathscr{A}(a)=\infty$ for every $a>\overline{\cI}_H(T)$.  It follows 
		that $\limsup_{\varepsilon\to0^+} \varepsilon^{-a}I_2(\varepsilon,T)=\infty$. 
		Together with \eqref{I_1:2}, this implies the second statement of the proposition.
	\end{compactenum}
\end{proof}		

\section{The linear heat equation}
\label{s4}
In this section we return to SPDEs,
and study the following linearization of the heat equation 
\eqref{SHE}:
\begin{equation}\left[\label{SHE:G}\begin{split}
	&\partial_t H = \rL H+ \dot{\rf} \quad\text{on $(0\,,\infty)\times\R^d$},\\
	&\text{subject to}\quad H(0)=0\text{ on $\R^d$}.
\end{split}\right.\end{equation}

The general theory of Dalang \cite{Dalang} ensures that \eqref{SHE:G} has a random-field solution
if and only if \eqref{Dalang} holds. Let us briefly review why this is the case in the present, simple context, of 
the heat equation \eqref{SHE:G}.

Because the fundamental solution to the operator $\partial_t-\rL$ is 
$p(t,\cdot)=\P\{X(t)\in\cdot\,\}$ [see for example the discussion around \eqref{p_t(F)}]
the stochastic heat equation \eqref{SHE:G} has a random-field solution if and only if the
stochastic convolution $p*\rf$ is a well-defined Gaussian random field. This means that
we have to have 
$p\in\cK$. Since $p(t)$ is a Borel probability measure for
every $t\ge0$, \eqref{NoZeros}, and \eqref{LKF1} 
together imply that \eqref{SHE:G} has a random-field solution if and only if
\[
	\int_0^T\d s\int_{\R^d}\mu(\d\xi)\left| \e^{-s\psi(\xi)}\right|^2
	= T\mu\{0\}+\frac{1}{2}\int_{\R^d\setminus\{0\}}
	\left(\frac{1-\e^{-2T\Re\psi(\xi)}}{\Re\psi(\xi)} \right)\mu(\d\xi)<\infty
	\qquad\text{for all $T>0$}.
\]

Because $\Re\psi\ge0$ and
$u^{-1}(1-\e^{-u})\asymp(1+u)^{-1}$ uniformly for all $u>0$, this shows that \eqref{Dalang} is necessary and sufficient for \eqref{SHE:G} to have a random-field solution.

In addition, Dalang's theory ensures that, when the solution exists as a random 
field, it automatically is continuous in $L^2(\Omega)$ in the following rather strong sense:
\[
 	\lim_{\varepsilon\to 0^+}\sup_{\substack{s,t\in[0,T]:\\ |s-t|\le\varepsilon}}
	\sup_{\substack{x,y\in\R^d:\\ \|x-y\|\le\varepsilon}}
	\E\left( |H(t\,,x) - H(s\,,y)|^2\right)=0\qquad\text{for all $T>0$}.
\]
In particular, Doob's separability theory ensures that under condition
\eqref{Dalang}, the random field $(t\,,x)\mapsto H(t\,,x)$ 
is Lebesgue measurable [more precisely, has a Lebesgue-measurable modification].
In this section we identify when exactly $H$ has locally H\"older-continuous sample functions.
Before we discuss those details, let us point out that,
as a consequence of the mentioned separability,
Kallianpur's zero-one law for Gaussian processes (Kallianpur \cite{Kallianpur};
see also Cambanis and Rajput \cite{CambanisRajput}) tells us that for all $s>0$ and $y\in\R^d$:
\begin{align}\label{0-1:law}
&\P\{H(\cdot\,,y)\text{ is locally H\"older continuous 
		over $\R_+$}\}=0\text{ or }1,\notag\\
	&\P\{H(s,\cdot)\text{ is locally H\"older continuous 
		over $\R^d$}\}=0\text{ or }1,\\
	&\P\{H\text{ is locally H\"older continuous 
		over $\R_+\times\R^d$}\}=0\text{ or }1\notag.
\end{align}
Therefore, we sometimes simply say that $H$ has [or does not have, depending on the case]
locally H\"older continuous samples when we really mean to say that
$H$ a.s.\ has [or a.s.\ does not have, depending on the case] a modification 
that is locally H\"older continuous. 

The goal of this section is to prove the following result. Throughtout we assume \eqref{Dalang}.
It might help to also remind that if $\indL>0$ then $\indH>0$, but the converse need not hold; 
see \eqref{indR<indH} and Proposition \ref{pr:ind:-}.
 

 \begin{theorem}\label{th:Holder:H}
 1.\ If $\indH>0$, then
	\begin{equation}\label{H:x}
		\P\left\{H(\cdot\,,x)\in C^{\alpha}_{\textit{loc}}(\R^d)\right\}=1\quad
		\text{for every $x\in\R^d$ and $0< \alpha< \tfrac12\indH$;}
	\end{equation}
	and if $\indL>0$, then 
	\begin{equation}\label{H:t}
		\P\left\{H\in C^{\alpha,\beta}_{\textit{loc}}(\R_+\times\R^d)\right\}=1\quad
		\text{for every $0<\alpha<\tfrac12\indH$ and $0<\beta<\indL$.}
	\end{equation}
	
	2.\ If $\indH=0$, then $t\mapsto H(t\,,0)$ is not H\"older continuous;
	and if $\indL=0$, then $x\mapsto H(t\,,x)$ is not H\"older continuous for every  $t>0$.
\end{theorem}

Note that Theorem \ref{th:heat} reduces to a weaker version of Theorem \ref{th:Holder:H}
when $g\equiv u_0\equiv0$.

As the first step of
the proof of Theorem \ref{th:Holder:H} we evaluate the indices $\cI_H$, $\underline{\cI}_H$,
and $\overline{\cI}_H$ in terms of the indices $\indH$, $\indL$ and $\cI_R(T)$ defined in \eqref{ind:H}, \eqref{ind:L}, \eqref{phi*f:indR}, respectively.

\begin{lemma}\label{lem:ind:H}
	$\cI_R(T)=2 \indL$ for all $T>0$,
	and $\cI_H=\underline{\cI}_H(T)=\overline{\cI}_H(T)=\indH$.
\end{lemma}

\begin{proof}
	Recall the transition functions of the process $X$ from \eqref{p_t(F)} and appeal to
	the non-degeneracy condition \eqref{NoZeros} and
	the L\'evy-Khintchine formula \eqref{LKF} to see that for every $T>0$ and $\eta>0$ fixed,
	\[
		\int_0^T\d s\int_{\R^d}\mu(\d\xi) \|\xi\|^{\eta} \ |\hat{p}(s\,,\xi)|^2
		= \int_0^T\d s\int_{\R^d}\mu(\d\xi)\ \|\xi\|^{\eta}\ \e^{-2s\Re\psi(\xi)}
		\asymp \int_{\R^d}\frac{\|\xi\|^\eta\,\mu(\d\xi)}{1+\Re\psi(\xi)}.
	\]
	This proves that $\cI_R(T)=2\indL$ for all $T\ge0$.
	
	Next, we observe that if $b\in(0\,,1)$, then
	\[
		\int_0^1\frac{\d s}{s^b}\int_{\R^d}\mu(\d\xi)\ 
		| \hat{p}(s\,,\xi)|^2 = \int_{\R^d}\mu(\d\xi)
		\int_0^1\frac{\d s}{s^b}\ \e^{-2s\Re\psi(\xi)} \asymp
		\int_{\R^d}\frac{\mu(\d\xi)}{
		[1+\Re\psi(\xi)]^{1-b}},
	\]
	thanks to \eqref{NoZeros},
	the Tonelli theorem, and Lemma \ref{lem:B} of the Appendix.
	This proves that $\cI_H=\indH$. 
	
	For the next stage of the proof, we observe that for every $t>0$,
	\begin{align*}
		\mathscr{T}(t) &:=\int_0^t\d s\int_{\R^d}\mu(\d\xi)\
			\left| \hat{p}(s + r \,,\xi)-\hat{p}(s\,,\xi)\right|^2
			= \int_0^t\d s\int_{\R^d}\mu(\d\xi)\
			\left|\e^{-s\psi(\xi)}  - \e^{-(s+r)\psi(\xi)}\right|^2\\
		&=\int_0^t\d s\int_{\R^d}\mu(\d\xi)\
			\left|1 - \e^{-r\psi(\xi)}\right|^2 \e^{-2s\Re\psi(\xi)}
			\asymp\int_{\R^d}\left|1 - \e^{-r\psi(\xi)}\right|^2 \frac{\mu(\d\xi)}{1+\Re\psi(\xi)}.
	\end{align*}
	Therefore, uniformly for all $\varepsilon>0$,
	\[
		\sup_{r\in[0,\varepsilon]}\mathscr{T}(r)\lesssim
		\sup_{r\in[0,\varepsilon]}
		\int_{\R^d}\left|1 \wedge r\psi(\xi) \right|^2 \frac{\mu(\d\xi)}{1+\Re\psi(\xi)}
		= \int_{\R^d}\left|1 \wedge \varepsilon\psi(\xi) \right|^2 \frac{\mu(\d\xi)}{1+\Re\psi(\xi)}.
	\]
	This readily implies the following for every $a>0$ fixed:
	\[
		\int_0^1\frac{\d\varepsilon}{\varepsilon^{1+a}}
		\sup_{r\in[0,\varepsilon]}\mathscr{T}(r)
		\lesssim\int_{\R^d}\frac{[1+|\psi(\xi)|^a]\,\mu(\d\xi)}{1+\Re\psi(\xi)}.
	\]
	Because of Dalang's condition \eqref{Dalang}, the preceding
	implies that
	$\underline{\cI}_H(T)\ge\indH$; see \eqref{phi*f:indHSS'} and \eqref{ind:H}.
	
	Next we prove that for every $a>0$ fixed,
	\begin{equation}\label{goal:H}
		\int_0^1\frac{\d\varepsilon}{\varepsilon^{1+a}} \mathscr{T}(\varepsilon)
		\gtrsim\int_{|\psi|\ge 2}\frac{|\psi(\xi)|^a\,\mu(\d\xi)}{1+\Re\psi(\xi)}.
	\end{equation}
	Because Dalang's condition holds, this will show that 
	$\overline{\cI}_H(T)\leq\indH$, and completes the proof of the lemma.
	
	First, we note that
	for all $\xi\in\R^d$ and $r\ge0$, 
	\begin{align*}
		\left|1 - \e^{-r\psi(\xi)}\right|^2  &= \left( 1 - \e^{-r\Re\psi(\xi)}\right)^2 + 2
			\left( 1 - \cos\left[ r|\Im\psi(\xi)|\right]\right) \e^{-r\Re\psi(\xi)}\\
		&\gtrsim \left( 1\wedge r^2[\Re\psi(\xi)]^2\right) +
			\left( 1 - \cos\left[ r|\Im\psi(\xi)|\right]\right) \e^{-r\Re\psi(\xi)}.
	\end{align*}
	It follows from a few lines of computation that
	\[
		\int_0^1\frac{\d\varepsilon}{\varepsilon^{1+a}}
		\mathscr{T}(\varepsilon)
		\gtrsim\int_{\R^d}\frac{|\Re\psi(\xi)|^a\,\mu(\d\xi)}{1+\Re\psi(\xi)}
		+ \mathscr{B},
	\]
	where
	\[
		\mathscr{B} = \int_0^1\frac{\d\varepsilon}{\varepsilon^{1+a}}
		\int_{\R^d}\left( 1 - \cos\left[\varepsilon|\Im\psi(\xi)|\right]\right) 
		\e^{-\varepsilon\Re\psi(\xi)}\, \frac{\mu(\d\xi)}{1+\Re\psi(\xi)}.
	\]
	Since the integrand of the above integral is nonnegative,
	\begin{align*}
		\mathscr{B}&\ge\int_{|\Im\psi(\xi)|\ge\Re\psi(\xi)\ge1}\frac{\mu(\d\xi)}{1+\Re\psi(\xi)}
			\int_0^1\frac{\d\varepsilon}{\varepsilon^{1+a}}
			\left( 1 - \cos\left[\varepsilon|\Im\psi(\xi)|\right]\right) 
			\e^{-\varepsilon\Re\psi(\xi)}\\
		&\ge\int_{|\Im\psi(\xi)|\ge\Re\psi(\xi)\ge1}\frac{|\Im\psi(\xi)|^a\,\mu(\d\xi)}{1+\Re\psi(\xi)}
			\int_0^{|\Im\psi(\xi)|}\frac{\d s}{s^{1+a}}(1-\cos s)
			\e^{-s},
	\end{align*}
	after a change of variables $[s=\varepsilon|\Im\psi(\xi)|$]. Therefore, we replace
	$\int_0^{|\Im\psi(\xi)|}(\,\cdots)$ with $\int_0^1(\,\cdots)$ to find that
	\[
		\mathscr{B} \gtrsim\int_{|\Im\psi(\xi)|\ge\Re\psi(\xi)\ge1}
		\frac{|\psi(\xi)|^a\,\mu(\d\xi)}{1+\Re\psi(\xi)},
	\]
	and hence
	\begin{align*}
		\int_0^1\frac{\d\varepsilon}{\varepsilon^{1+a}}
			\mathscr{T}(\varepsilon) &\gtrsim
			\int_{\R^d}\frac{|\Re\psi(\xi)|^a\,\mu(\d\xi)}{1+\Re\psi(\xi)}
			+\int_{|\Im\psi(\xi)|\ge\Re\psi(\xi)\ge1}
			\frac{|\psi(\xi)|^a\,\mu(\d\xi)}{1+\Re\psi(\xi)}\\
		&\gtrsim\int_{|\psi(\xi)|\ge2}
			\frac{|\psi(\xi)|^a\,\mu(\d\xi)}{1+\Re\psi(\xi)}.
	\end{align*}
	Since $\psi$ is continuous it is locally bounded. Because $\mu$ is locally finite,
	this proves 
	that $\overline{\cI}_H(T)\le\indH$ and completes the proof (see \eqref{phi*f:indHSS'} and \eqref{ind:H}).
\end{proof}

We will also have need for the following classical fact from real analysis.

\begin{lemma}[The Paley-Zygmund inequality {\cite[Lemma $\gamma$]{PZ}}]\label{lem:PZ}
	If $V$ is a non-negative, mean-one random variable in $L^2(\Omega)$, then
	$\P \{ V \ge \lambda  \} \ge  (1-\lambda)^2/\E(V^2)$ for every $\lambda\in(0\,,1)$.
\end{lemma}

We are ready for the following.

\begin{proof}[Proof of Theorem \ref{th:Holder:H}]
	First, suppose that $\indH >0$. 
	Choose and fix an arbitrary 
	number $a\in(0\,,\indH)$ and appeal to Lemma \ref{lem:ind:H}
	and Proposition \ref{pr:phi*f:t} 
	in order to see that
	\[
		\E\left( \left| H(t+\varepsilon\,,x) - H(t\,,x)\right|^2\right) \lesssim \varepsilon^{a},
			\]
	uniformly for all  $t\in[0,T]$, $\varepsilon\in(0\,,1)$, and $x\in\R^d$.
	[It might help to also recall that the law of $t\mapsto H(t\,,x)$ does not depend on $x$.]
	Because 
	\begin{equation}\label{L^n(O)}
		\E(W^n) = (n-1)!! [ \Var(W)]^{n/2}
		\qquad\text{for every even integer  $n\ge1$,}
	\end{equation}
	and for every centered normally distributed random variable $W$
	on $\R$, it follows that for every even integer $n$ fixed,
	\begin{equation}\label{moment:n}
		\adjustlimits\sup_{t\ge0}\sup_{x\in\R^d}
		\E\left( \left| H(t+\varepsilon\,,x) - H(t\,,x)\right|^n\right) 
		\lesssim\varepsilon^{na/2},
	\end{equation}
	uniformly for every $\varepsilon\in(0\,,1)$.
	Therefore, the Kolmogorov continuity theorem implies
	that $\P\{H(\cdot\,,x)\in C^\alpha_{\textit{loc}}(\R^d)\}=1$ for every $x\in\R^d$
	and $\alpha\in(0\,,a/2)$. Since $a\in(0\,,\indH)$ is otherwise arbitrary, this proves
	\eqref{H:x}.
	\smallskip
	
	We verify \eqref{H:t} next. With this in mind, let us assume that
	$\indL>0$, equivalently $\cI_R(T)>0$ (by Lemma \ref{lem:ind:H}).
	Choose and fix an arbitrary 
	number $b\in(0\,,\cI_R(T))$ and appeal to 
	Proposition \ref{pr:phi*f:x} 
	in order to see that
	\begin{equation}\label{heat:moment:0}
		\adjustlimits\sup_{t\in[0,T]}\sup_{x\in\R^d}
		\E\left( \left| H(t\,,x+h) - H(t\,,x)\right|^2\right) \lesssim \|h\|^{b},
	\end{equation}
	uniformly for every $h\in\R^d$ that satisfies $\|h\|\le1$.
	Therefore, \eqref{L^n(O)} implies that for every integer $n\ge1$,
	\begin{equation}\label{heat:moment:00}
		\adjustlimits\sup_{t\in[0,T]}\sup_{x\in\R^d}
		\E\left( \left| H(t\,,x+h) - H(t\,,x)\right|^n\right) \lesssim\|h\|^{nb/2}
		\end{equation}
	uniformly for every $h\in\R^d$ that satisfies $\|h\|\le1$.
	Our assumption that $\indL>0$ implies also that $\indH>0$; see \eqref{indR<indH}.
	Therefore, \eqref{moment:n} also holds. Consequently, for every integer $n\in\N$ 
	and all $a\in(0\,,\indH)$ and $\bar{b}\in(0\,,\indL)$, all fixed,
	\begin{equation}\label{heat:moment}
		\E\left( \left| H(t\,,x) - H(s\,,y)\right|^n\right) \lesssim
		|t-s|^{an/2}+\|x-y\|^{\bar{b}n},
	\end{equation}
	uniformly for all tuples $(t\,,x)$ and $(s\,,y)$ on a (given and fixed) bounded rectangle.
	Because $a\in(0\,,\indH)$ and $\bar{b}\in(0\,,\indL)$ are otherwise
	arbitrary, Kolmogorov's continuity theorem implies \eqref{H:t}.
	
	Next, we suppose that $\P\{ H(\cdot\,,0)\in C^\gamma_{\textit{loc}}(\R_+)\}=1$
	for some $\gamma>0$ and aim to prove that $\indH >0$. 	
	Note that
	\[
		\P\left\{\adjustlimits\lim_{\varepsilon\to0^+}\sup_{s\in(0,1)}
		\frac{|H(s+\varepsilon\,,0) - H(s\,,0)|}{\varepsilon^{\theta}}=0\right\}=1
		\qquad\text{for every $\theta\in(0\,,\gamma)$}.
	\]
	Among other things, this implies the existence of two numbers $\varepsilon_0\in(0\,,1)$
	and $K>1$ such that
	\begin{equation}\label{A0}
		\adjustlimits\inf_{\varepsilon\in(0,\varepsilon_0)}\inf_{s\in(0,1)}
		\P\left\{ |H(s+\varepsilon\,,0)-H(s\,,0)|^2 < K\varepsilon^{2\theta} \right\}>
		1 - \frac{1}{2\pi}.
	\end{equation}
	Apply Lemma \ref{lem:PZ} to the random variable 
	$V:=  |H(s+\varepsilon\,,0)-H(s\,,0)|/ \| H(s+\varepsilon\,,0)-H(s\,,0)\|_1$ to see that
	\begin{align*}
		&\P\left\{ |H(s+\varepsilon\,,0)-H(s\,,0)|  \ge \frac12
			\| H(s+\varepsilon\,,0)-H(s\,,0)\|_1 \right\}\\
		&\hskip2.4in\ge \frac{\| H(s+\varepsilon\,,0)-H(s\,,0)\|_1^2}{%
			4\| H(s+\varepsilon\,,0)-H(s\,,0)\|_2^2} = \frac{1}{2\pi},
	\end{align*}
	thanks to the elementary fact that $\E|W|=\sqrt{2/\pi}\ \E(|W|^2)$ for every centered normal
	random variable on $\R$. 
	
	Once again we apply Lemma \ref{lem:PZ}, this time with 
	$V:=  |H(s+\varepsilon\,,0)-H(s\,,0)|^2/ \| H(s+\varepsilon\,,0)-H(s\,,0)\|_2^2$ 
	and then we appeal to \eqref{L^n(O)} with $n=4$, in order to find that
	\begin{equation}\label{B}
		\P\left\{ |H(s+\varepsilon\,,0)-H(s\,,0)|^2  \ge 
		\frac{\|H(s+\varepsilon\,,0)-H(s\,,0)\|_2^2}{2\pi} \right\}
		\ge\frac{1}{2\pi}.
	\end{equation}
	
	If $E_1$ and $E_2$ are two events on the same probability 
	space and satisfy $\P(E_1)+\P(E_2)>1$,
	then $\P(E_1\cap E_2) \ge \P(E_1)+\P(E_2)-1>0$, thanks to Bonferroni's inequality. 
	We apply this fact with
	$E_1$ and $E_2$ denoting respectively the event in \eqref{A0} and \eqref{B} in order to see that
	\begin{equation}\label{C}
		\sup_{s\in(0,1)}
		\E\left( \left| H(s+\varepsilon\,,0)-H(s\,,0) \right|^2 \right) \le 2\pi K\varepsilon^{2\theta},
	\end{equation}
	uniformly for every $\varepsilon\in(0\,,\varepsilon_0)$. Because this is true for every
	$\theta\in(0\,,\gamma)$, the second half of Proposition \ref{pr:phi*f:t} 
	implies that $\cI_H\wedge\underline{\cI}_H(T)>0$. Equivalently,
	$\indH >0$, because of Lemma \ref{lem:ind:H}.
	 
	Similarly, we may use the second half of Proposition \ref{pr:phi*f:x}
	together with the Paley-Zygmund inequality (Lemma \ref{lem:PZ}) to prove that if
	$x\mapsto H(t\,,x)$ is in $C^\beta(\R^d)$ for some $\beta>0$ and $t>0$,
	then $\cI_R(T)>0$. By Lemma \ref{lem:ind:H} this is equivalent to 
	$\indL>0$. This concludes the proof of the theorem.
\end{proof}

\section{The linear wave equation}
\label{s5}

In this section we study the regularity and well-posedness of the 
stochastic wave equation:
\begin{equation}\label{SWE:G}\left[\begin{split}
	&\partial_t^2 W = \rL W  + \dot{\rf}\quad\text{on $(0\,,\infty)\times\R^d$},\\
	&\text{subject to}\quad W(0)=\partial_t W(t)=0\quad\text{on $\R^d$}.
\end{split}\right.\end{equation}
In order to be able to accomplish this goal we will restrict attention to the symmetric
case; that is, we assume that the
underlying L\'evy process $X$ has the same law as the L\'evy process $-X$.
This symmetry condition is equivalent to the following Fourier-analytic property:
\begin{equation}\label{Sym}
	\psi(\xi)=\psi(-\xi)=\Re\psi(\xi)\ge0\text{ for all $\xi\in\R^d$}.
\end{equation}
The restriction to $\cS(\R^d)$  of the generator $\rL$  of a symmetric L\'evy process is self-adjoint on $L^2(\R^d).$
Furthermore, in this case the characteristic exponent of $X$ reduces to 
\begin{equation}\label{LKF10}
	\psi(\xi) = \tfrac12\xi\cdot\bA\xi + \int_{\R^d}
	\left[1-\cos(y\cdot\xi)\right]\nu(\d y)\qquad
	\text{for all $\xi\in\R^d$}
\end{equation}
(see \eqref{LKF}).  

Throughout this section we also assume that 
\begin{equation}\label{RL}
	\lim_{\|\xi\|\to\infty}\psi(\xi)=\infty,
\end{equation}
which turns out to be a non-degeneracy condition for the law of the random vector $X(1)$.
In Remark \ref{rem5.1} below we discuss further this hypothesis.

For every $\xi\in\R^d$, define
 \[
	\hat{\varphi}(t\,,\xi) = t\sinc\left( t\sqrt{\psi(\xi)}\right)\quad [t\ge 0],
\]
where $\sinc(a)=\sin(a)/a$ for $a\neq0$ and $\sinc(0)=1$.\footnote%
{Some
	authors use instead the slightly different
	scaling $\sin(\pi a)/(\pi a)$ [$a\neq 0$] for the sinc function; see for example
	Baumann and Stenger \cite{BaumannStenger}.}	
 Thanks to \eqref{RL}  and the continuity of $\psi$ --ensured by \eqref{LKF10}--  one can check that, for any $t>0$,  the mapping $\xi\mapsto \hat{\varphi}(t\,,\xi)$ is a tempered distribution; in fact it is a pseudo measure \cite{Benedetto,Kahane}.
 
 Notice that for any fixed $\xi\in\R^d$, the function $t\mapsto \hat{\varphi}(t\,,\xi)$ satisfies the second order linear differential equation
 \[
	 \frac{\d^2}{\d t^2} \hat{\varphi}(t\,,\xi) + \psi(\xi)\hat{\varphi}(t\,,\xi) = 0,\quad t\ge 0,
 \]
subject to initial conditions
$\hat{\varphi(0,\xi)}=0$, $\frac{\d}{\d t} \hat{\varphi}(0\,,\xi)=1$.
Thus, we set
$\varphi(t\,,\cdot) := \F^{-1}\hat{\varphi}(t\,,\cdot)$, 
where $\F^{-1}$ denotes the inverse Fourier transform operator
(acting on the spatial variable of $\varphi$), 
in order to see that $\varphi\in \mathcal{S}^\prime(\R^d)$, and $\varphi$ is the fundamental solution of the linear wave operator $\partial^2_t-\rL$. 

Referring to Section \ref{sec:conv}, the random-field solution to the wave 
equation \eqref{SWE:G} is the Gaussian process defined by the stochastic convolution 
$\varphi*\rf=\{(\varphi*\rf)(t\,,x);\, t\ge0,x\in\R^d\}$. 
This process is well defined if and only if $\varphi\in \mathcal{K}$.
Since \eqref{RL} implies that $\{\psi<1\}$ is bounded, in the context of this section, it follows that
\[
	\varphi\in\cK
	\quad\text{if and only if}\quad
	\int_0^T t^2\,\d t\int_{\psi(\xi)\ge 1}\mu(\d\xi)
	\left[\sinc\left( t\sqrt{\psi(\xi)}\right)\right]^2 <\infty\quad\text{for all $T>0$}.
\]
From the identity
$\int_0^c \sin^2(x)\,\d x =\frac12 c [1 - \tfrac12\sinc(2c)]$, valid for every $c>0$, we deduce
\begin{equation}\label{int:sinc}
	\int_0^T t^2[\sinc(tK)]^2\,\d t = 
	\frac {T}{2K^2}\left[ 1-\sinc(2KT)\right]
	\asymp\frac{T}{1+K^2}, 	
	\end{equation}
uniformly for all $K,T\ge0$ such that $KT\ge1$. Consequently, $\varphi\in\cK$ if and only if Dalang's condition \eqref{Dalang} holds; because of \eqref{Sym} it reads
\[
\int_{\R^d} \frac{\mu(\d\xi)}{1+\psi(\xi)} < \infty.
\]
Remember that we are assuming \eqref{Dalang} throughout the paper.
  \smallskip

Condition  \eqref{RL} has relations to other parts of mathematics, and is
worthy of inspection in its own right. Therefore, we pause to include some remarks about \eqref{RL}.
Throughout the following remarks we make extensive references to \eqref{p_t(F)}.
Also, we recall that a finite Borel measure on $\R^d$ is \emph{Rajchman}
if its Fourier transform vanishes at infinity \cite{Lyons}. Thanks to
the L\'evy-Khintchine formula \eqref{LKF10},
we can interpret \eqref{RL} as saying that $p(t)$
is a Rajchman measure for some $t>0$. 

\begin{remark}
\label{rem5.1}
	\begin{compactenum}
		\item If $p(t)$ is absolutely continuous for some $t>0$ then \eqref{LKF10}
			implies \eqref{RL} thanks to the Riemann-Lebesgue lemma. In particular, 
			if $\exp\{-t\psi\}\in L^1(\R^d)$ for some $t>0$, then \eqref{RL} holds;
		\item In the case that $p(t)$ is radially symmetric for all $t>0$ and 
			$d\ge 2$, Zabczyk \cite[p.\ 245]{Zabczyk} 
			has proven that \eqref{RL} holds iff every excessive function of $X$ 
			is lower semicontinuous. Moreover \cite[Example (4.6)]{Zabczyk}, either $X$ is Poisson, or  
			$p(t\,,\d x)\ll\d x$, in which case
			\eqref{RL} holds automatically.
		\item If $\bA$ is non singular, 
			then $\psi(\xi)\gtrsim \|\xi\|^2$ uniformly
			for all $\xi\in\R^d$, whence  \eqref{RL} holds.
		\item Since $\psi\le 2\nu(\R^d)$ on the null space of $\bA$ [see \eqref{LKF10}] 
			if $\bA$ is singular and $\nu(\R^d)<\infty$, then \eqref{RL} fails.
		\item Suppose the absolutely continuous part of $\nu$ is infinite; that is, 
			$h=\d\nu/\d x\not\in L^1(\R^d)$. We can deduce from \eqref{LKF10} that
			$\psi(\xi) \ge \int_{S(\delta)}[1-\cos(y\cdot\xi)]h(y)\,\d y$ 
			for every $\xi\in\R^d$ and $\delta>0$,
			where $S(\delta)=\{y\in\R^d:\,\|y\|>\delta\}$.
			Thus, $\liminf_{\|\xi\|\to\infty}\psi(\xi)
			\ge \lim_{\delta\to0}\int_{S(\delta)}h=\infty$ by the Riemann-Lebesgue lemma, 
			and hence \eqref{RL} fails.
	\end{compactenum}
\end{remark}

We are ready to study the regularity
properties of the solution, assuming \eqref{Dalang}.
It might help to first recall \eqref{indR<indH}.

\begin{theorem}\label{th:Holder:W}

1.\ If $\indH>0$, then
\begin{equation}
\label{Holder:W-1}
	\P\left\{W(\cdot\,,x)\in C^\alpha_{\textit{loc}}(\R^d)\right\}=1 \
	\text{for every}\ x\in\R^d\ {\text and} \ 0< \alpha< \tfrac12\wedge\indH;
	\end{equation}
	and if $\indL>0$, then 
\begin{equation}
	\label{Holder:W-2}
	\P\left\{W\in C^{\alpha,\beta}_{\textit{loc}}(\R_+\times\R^d)\right\} =1\
	\text{for every}\  0< \alpha< \tfrac12\wedge\indH\ {\text and}\  0<\beta<\indL.
	\end{equation}
	
2.\ If $\indH=0$, then $t\mapsto W(t\,,0)$ is not H\"older continuous;
	and if $\indL=0$, then $x\mapsto W(t\,,x)$ is not H\"older continuous for every $t>0$.
\end{theorem}

Theorem \ref{th:Holder:W} is proved using the same argument as 
Theorem \ref{th:Holder:H} was, except the following index computation
replaces the role of Lemma \ref{lem:ind:H}.

\begin{lemma}\label{lem:ind:W}
	$\cI_R(T) = 2 \indL$ for all $T>0$, $\cI_H=1$, and
	$\underline{\cI}_H(T)=\overline{\cI}_H(T)=2\indH$.
\end{lemma}

\begin{proof}
	We proceed in stages.
	
	\emph{Stage 1.} Because $|\sinc(x)|\le 1\wedge |x|^{-1}$ for all $x\in\R\setminus\{0\}$ 
	and $1\wedge z^{-1}\asymp (1+z)^{-1}$ uniformly for all $z\in\R_+$,
	\begin{equation}
	\label{lem:ind:W-1}
		\int_0^T\d s\int_{\R^d}\mu(\d\xi)\ |\hat{\varphi}(s\,,\xi)|^2\|\xi\|^{2a}
		\le \int_0^Ts^2\,\d s\int_{\R^d}\mu(\d\xi)\left(
		1\wedge \frac{1}{\psi(\xi)}\right)\|\xi\|^{2a}
		\asymp\int_{\R^d}\frac{\|\xi\|^{2a}\,\mu(\d\xi)}{1+\psi(\xi)}.
	\end{equation}
	This proves that $\cI_R(T)\ge2\indL$ for all $T>0$.
	
	The first inequality in \eqref{lem:ind:W-1} can be reversed. Indeed, by Tonelli's theorem and \eqref{int:sinc},
	\begin{equation*}
		\int_0^T\d s\int_{\R^d}\mu(\d\xi)\ |\hat{\varphi}(s\,,\xi)|^2\|\xi\|^{2a}
		\gtrsim \int_{\psi\ge1/T^2} \,\frac{\|\xi\|^{2a}\mu(\d\xi)}{1+\psi(\xi)}
		\qquad\text{uniformly for all $T>0$}.
	\end{equation*}
	 The condition \eqref{RL} ensures that $\{\xi\in\R^d:\ \psi(\xi)\le C\}$ is compact
	 for every choice of $C>0$. Hence, 
	$\int_{\psi\le1/T^2} \|\xi\|^{2a}(1+\psi(\xi))^{-1}\,\mu(\d\xi)$ is finite for every $a>0$. Therefore,
	the preceding two displays
	show that, for any $a,T>0$,
	\[
		\int_0^T\d s\int_{\R^d}\mu(\d\xi)\ |\hat{\varphi}(s\,,\xi)|^2\|\xi\|^{2a}<\infty
		\quad\text{if and only if}\quad
		\int_{\R^d}\frac{\|\xi\|^{2a}\mu(\d\xi)}{1+\psi(\xi)}<\infty.
	\]
	This is more than enough to show that $\cI_R(T)\le 2\indL$ for all $T>0$, which in turn proves
	that $\cI_{R}(T)= 2\indL$ for all $T>0$, thanks to the first portion of the proof.
		
	\emph{Stage 2.}
	Similarly as in \eqref{lem:ind:W-1}, for every $b\in(0\,,3)$,
	\[
		\int_0^1\frac{\d s}{s^b}\int_{\R^d}\mu(\d\xi)\ 
		| \hat{\varphi}(s\,,\xi)|^2 \le\int_0^1 \d s\ s^{2-b}\,\int_{\R^d}\mu(\d\xi)
		\left( 1\wedge \frac{1}{\psi(\xi)}\right)
		\asymp\int_{\R^d}\frac{\mu(\d\xi)}{1+\psi(\xi)}.
	\]
	We are assuming that $\int_{\R^d}(1+\psi(\xi))^{-2} \mu(\d\xi)<\infty$. 
	Thus, we see from the definition of $\cI_H$ (see \eqref{phi*f:indHSS'}) 
	that $b\le \cI_H\le 1$ for every $b\in(0\,,1)$. It follows that $\cI_H=1$.
	
	\emph{Stage 3.}
	Next we observe that $|\sin x - \sin y|\le 2\wedge|y-x|$ for all $x,y\in\R$, and hence
	\begin{align*}
		&\int_0^1\frac{\d\varepsilon}{\varepsilon^{1+b}}\sup_{r\in[0,\varepsilon]}
			\int_0^t\d s\int_{\R^d}\mu(\d\xi)\
			\left| \hat{\varphi}(s + r \,,\xi)-\hat{\varphi}(s\,,\xi)\right|^2\\
		&\lesssim\int_{\R^d}\frac{\mu(\d\xi)}{\psi(\xi)}
			\int_0^1\frac{\d\varepsilon}{\varepsilon^{1+b}}
			\left(1\wedge\varepsilon^2\psi(\xi)\right)\\
			&=  \frac{\mu\{\psi\le1\}}{2-b} + 
			\left( \frac{1}{2-b} + \frac1b\right)
			\int_{\psi>1} \frac{\mu(\d\xi)}{[\psi(\xi)]^{1-(b/2)}}-
			\frac{1}{b}\int_{\psi>1} \frac{\mu(\d\xi)}{\psi(\xi)}\\
		&\asymp \int_{\R^d}\frac{|\psi(\xi)|^{b/2}\,\mu(\d\xi)}{1+\psi(\xi)}\qquad
			\text{for every fixed $b\in(0\,,2)$}.
	\end{align*}
	This proves that $\underline{\cI}_H(T)\ge2\indH$.

 For the complementary bound we notice that
	for every $b,t>0$,
	\begin{align*}
		\mathscr{C}&:=\int_0^1\frac{\d\varepsilon}{\varepsilon^{1+b}}
			\int_0^t\d s\int_{\R^d}\mu(\d\xi)\
			\left| \hat{\varphi}(s + \varepsilon \,,\xi)-\hat{\varphi}(s\,,\xi)\right|^2\\
		&\gtrsim\int_{\psi\ge 16 K(t)}\frac{\mu(\d\xi)}{1+\psi(\xi)}\int_0^t \d s
			\int_0^1\frac{\d\varepsilon}{\varepsilon^{1+b}}
			\left| \sin\left((s + \varepsilon )\sqrt{\psi(\xi)}\right)
			- \sin\left(s\sqrt{\psi(\xi)}\right)\right|^2\\
		&\ge\int_{\psi\ge 16 K(t)}\frac{[\psi(\xi)]^{(b-1)/2}\,\mu(\d\xi)}{1+\psi(\xi)}
			\int_0^{t\sqrt{\psi(\xi)}}\d u
			\int_0^{\sqrt{\psi(\xi)}}\frac{\d v}{v^{1+b}}
			\left| \sin(u+v) - \sin(u)\right|^2,
	\end{align*}
	where $K(t) := (\pi/t)^2 \vee 1.$
	
	By Taylor's theorem, $\sin(u+v) - \sin(u) \ge v/\sqrt 2$ uniformly for all
	\[
		v\in\left[0\,,\frac\pi8\right]
		\quad\text{and}\quad
		u\in \bigcup_{n=0}^\infty J_n
		\text{ where }J_n=\left[2 n\pi \,, 2n\pi + \frac\pi8\right].
	\]
	Therefore, for every $b,t>0$,
	\begin{align*}
		\mathscr{C}&\gtrsim\int_{\psi\ge 16K(t)}
			\frac{[\psi(\xi)]^{(b-1)/2}\,\mu(\d\xi)}{1+\psi(\xi)}
			\hskip-.3in\sum_{\substack{n\in\Z_+:\\ 2\pi n + (\pi/8)\le t\sqrt{\psi(\xi)}}}
			\int_{J_n}\d u
			\int_0^{\pi/8}\frac{\d v}{v^{1+b}}
			\left| \sin(u+v) - \sin(u)\right|^2\\
		&\gtrsim\int_{\psi\ge 16K(t) }\frac{[\psi(\xi)]^{(b-1)/2}}{1+\psi(\xi)}
			\max\left\{n\in\Z_+:\, 2\pi n + \frac\pi8\le t\sqrt{\psi(\xi)}\right\}\mu(\d\xi).
	\end{align*}
	Because of the choice of $K(t)$,   we have
	$t\sqrt{\psi(\xi)}\ge 4t\sqrt{K(t)}\ge4\pi$
	on the set $\{\xi\in\R^d:\, \psi(\xi)\ge 16 K(t)\}$; therefore,
	\begin{align*}
		\max\left\{n\in\Z_+:\, 2\pi n + \frac\pi8\le t\sqrt{\psi(\xi)}\right\}
		&\ge \max\left\{n\in\Z_+:\, n \le \frac{15t}{32\pi}  \sqrt{\psi(\xi)}\right\}\\
		&\ge \frac{15t}{32\pi} \sqrt{\psi(\xi)}-1\ge \frac{7t}{32\pi} \sqrt{\psi(\xi)}.
	\end{align*}
	Thus, for every fixed $b,t>0$,
	\[
		\mathscr{C} \gtrsim\int_{\psi\ge 16K(t)} [\psi(\xi)]^{b/2}[1+\psi(\xi)]^{-1}\,\mu(\d\xi).
	\]
	Because \eqref{RL} implies that $\int_{\psi\le 16 K(t) }[\psi(\xi)]^{b/2}[1+\psi(\xi)]^{-1}\,\mu(\d\xi)<\infty$,
	 it follows that
	\[
		\mathscr{C}<\infty\quad \Longrightarrow \quad
		\int_{\R^d} [\psi(\xi)]^{b/2} [1+\psi(\xi)]^{-1}\,\mu(\d\xi) <\infty.
	\]
	This shows that $\overline{\cI}_H(T)\le2\indH$ (see \eqref{ind:H} and \eqref{phi*f:indHSS'}), 
	and completes the proof.
\end{proof}

\section{Nonlinear SPDEs}
\label{s6}

This section is about extending the results of Section \ref{s4} to the stochastic heat equation with a nonlinear drift term. In the last part, a brief discussion on the difficulties for a similar extension concerning the wave equation is included.

Consider the SPDE  \eqref{SHE}.
We suppose that the function $g$,
 defined on $\R_+\times \R^d\times\R\times \Omega$ with values in $\R$,
is $\mathcal{B}_{\R_+\times \R^d\times \R}\otimes \mathcal{F}$-measurable 
and adapted to the natural filtration $(\mathcal{F}_t)_{ t\ge 0}$ associated to the noise $\dot F$.
That is, for every 
number fixed $t\ge 0$, $(x\,,z\,;\omega)\mapsto g(t\,,x\,,z\,;\omega)$ 
is $\mathcal{B}_{\R^d\times \R}\otimes \mathcal{F}_t$-measurable. Furthermore, we fix $T>0$ and assume the following:\\

\noindent{\em Global Lipschitz continuity.} There exists a constant $c_1(T)>0$ such that for all 
$(t\,,x\,;\omega)\in [0\,,T]\times \R^d\times \Omega$ and $z_1, z_2\in \R$,
\[
	|g(t\,,x\,,z_1\,;\omega) - g(t\,,x\,,z_2\,;\omega)| \le c_1(T) |z_1 - z_2|.
\]
\noindent{\em Uniform linear growth.} There exists a constant $c_2(T)>0$ such that for all 
$(t\,,x\,;\omega)\in  [0\,,T]\times \R^d\times \Omega$ and $z\in\R$,
\[
	|g(t\,,x\,,z\,;\omega)|\le c_2(T) (1+|z|).
\]
In the sequel, we write $c_T = c_1(T)\vee c_2(T)$.

As regards the initial condition, we assume that $u_0=\{u_0(x);\ x\in\R^d\}$ is a 
random field that is independent of the noise $\dot F$, and satisfies some assumptions made explicit in the specific statements.

As per the theory of Walsh \cite{Walsh},
the random-field solution to \eqref{SHE}  is defined as the unique solution
to the following integral equation:  For any $t>0$ and $x\in\R^d$,
\begin{equation}\label{mild:heat}
	u(t\,,x) = (p(t)*u_0)(-x) + \int_0^t\d s\int_{\R^d}p(t-s\,,\d y)\
	g(s\,,y\,,\,u(s\,,y-x)) + H(t\,,x),\
\end{equation}
a.s., where 
$p(t\,,\cdot) = \P\{X(t)\in\cdot\,\}$ -- see \eqref{p_t(F)} -- and 
$H(t\,,x) = (p*F)(t\,,x)$ is the random-field solution to \eqref{SHE:G}.
 It might help to also recall that the first term on the right-hand side
-- that is $(p(t)*u_0)(-x)=\E (u_0(x+X(t))$ -- defines
the semigroup of $X$.

If the condition $\sup_{x\in\R^d}\E(|u_0(x)|^2)< \infty$ holds, 
then we may apply a well-known fixed point argument to prove 
the existence of a jointly measurable and adapted process 
$\{u(t\,,x);\,t\in[0\,,T]\times \R^d\}$ that satisfies \eqref{mild:heat} for every
$(t\,,x)\in[0\,,T]\times \R^d$ a.s., and
\begin{equation}\label{integrability}
	\sup_{t\in[0,T]}\sup_{x\in\R^d} \|u(t\,,x)\|_2<\infty.
\end{equation}
Moreover, if  $\sup_{x\in\R^d}\E(|u_0(x)|^{k})< \infty$ for some $k\ge 2$, then \eqref{integrability} holds with the norm $\|\cdot\|_2$ there replaced by $\|\cdot\|_{k}$ (see e.g. \cite{Dalang} for an approach to the proof).
\smallskip

Suppose first that there exists a non-random constant $a_0\in\R$ such that
$g(z)=a_0$ for all $z\in\R$. Then, \eqref{mild:heat}
yields $u(t\,,x) = (p(t)*u_0)(-x) + a_0t + H(t\,,x)$.
It now follows from Proposition \ref{pr:heat} that $u$ is locally H\"older continuous iff $H$ is. 
Consequently, Theorem \ref{th:Holder:H} tells us that $u$ is locally H\"older continuous iff $\indH>0$
when $g$ is constant. This proves half of Theorem \ref{th:heat}. Thanks to the Kolmogorov continuity theorem,
the following two propositions (Propositions \ref{pr:Var:u:t} and \ref{pr:Var:u:x}) together imply 
the sufficiency of the condition $\indH>0$ in the general case, and complete the proof
of Theorem \ref{th:heat}. 
\smallskip

\begin{lemma}\label{lem:Var:u:t}
	Let $T>0$ and $k\ge 2$, and assume that
	\begin{equation}\label{lem:Var:u:t-1}
		K(T):=\sup_{t\in[0,T]}\sup_{x\in\R^d} \|u(t\,,x)\|_k<\infty.
	\end{equation}
	Set
	\[
		I(t\,,x)= \int_0^t\d s\int_{\R^d}p(t-s\,,\d y)\ g(s\,,y\,,u(s\,,y-x))
		\qquad\text{for every $t\ge 0$ and $x\in\R^d.$}
	\]
	Then, there exists a constant $C_T>0$ such that for every $t\in(0\,,T]$ and  $\varepsilon\ge 0$, 
	\begin{equation}\label{lem:Var:u:t-2}
		\sup_{x\in\R^d} \|I(t+\varepsilon\,,x)-I(t\,,x)\|_{k} \le 
		C_T \left(\varepsilon +\int_0^t \sup_{y\in\R^d} \|u(s+\varepsilon\,,y) - u(s\,,y)\|_k\ \d s\right).
	\end{equation}
\end{lemma}

\begin{proof}
	Write $I(t+\varepsilon\,,x)-I(t\,,x)= I_1(t\,,x\,;\varepsilon)+I_2(t\,,x\,;\varepsilon)$, where
	\begin{align*}
		I_1(t\,,x\,;\varepsilon)&= \int_t^{t+\varepsilon}\d s\int_{\R^d}\d y\ 
			p(t+\varepsilon-s\,,\d y)\,
			g (s\,,y\,,u(s\,,y-x)),\\
		I_2(t\,,x\,;\varepsilon) &= \int_0^t \d s\int_{\R^d} 
			\left[ p(t+\varepsilon-s\,,\d y)-p(t-s\,,\d y)\right] g(s\,,y\,,u(s\,,y-x)).
	\end{align*}
	To estimate $I_1(t\,,x\,;\varepsilon)$, we apply
	Minkowski's inequality and use the uniform linear growth of the function $g$. In
	this way we find that
	\begin{equation}\label{I3}
	\|I_1(t\,,x\,;\varepsilon)\|_k \le c_T\int_t^{t+\varepsilon}\d s\int_{\R^d}
		p(t+\varepsilon-s\,,\d y)\left( 1+\|u(s\,,y+x)\|_k\right)\le c_TK(T)\varepsilon,
	\end{equation}
	uniformly for all $t\in(0\,,T]$ and $x\in\R^d$.
	
	Next, we estimate $I_2(t\,,x\,;\varepsilon)$ by writing
	$I_2(t\,,x\,;\varepsilon) = I_{2,1}(t\,,x\,;\varepsilon)+ I_{2,2}(t\,,x\,;\varepsilon)+
	I_{2,3}(t\,,x\,;\varepsilon)$, where
	\begin{align*}
		I_{2,1}(t\,,x\,;\varepsilon) &= \int_0^\varepsilon \d s\int_{\R^d} p(t+\varepsilon-s\,,\d y)\, 
			g(s\,,y\,,u(s\,,y-x)),\\
		I_{2,2}(t\,,x\,;\varepsilon)&= \int_{t-\varepsilon}^t \d s\int_{\R^d} p(t-s\,,\d y) \,
			g(s\,,y\,,u(s\,,y-x)),\\
		I_{2,3}(t\,,x\,;\varepsilon) &= \int_0^{t-\varepsilon} \d s\int_{\R^d} p(t-s\,,\d y) 
			\left[ g(s\,,y\,,u(s+\varepsilon\,,y-x)) - g(s\,,y\,,u(s\,,y-x))\right].
	\end{align*}
	The same arguments that were used to estimate the term $I_1(t\,,x\,;\varepsilon)$ yield
	\begin{equation}\label{I21;22}
		\|I_{2,1}(t\,,x\,;\varepsilon)\|_k + \|I_{2,2}(t\,,x\,;\varepsilon))\|_k \le  c_TK(T)\varepsilon,
	\end{equation}
	valid uniformly for all $t\in(0\,,T]$ and $x\in\R^d$. 
	Finally, because of the Lipschitz continuity property of $g$,
	\[
		\|I_{2,3}(t\,,x\,;\varepsilon)\|_k \le c_T \int_0^t \sup_{y\in\R^d}
		\left\| u(s+\varepsilon\,,y)-u(s\,,y) \right\|_k\,\d s,
	\]
	uniformly for  all $t\in(0\,,T]$ and $x\in\R^d$.
	Combine this with \eqref{I3}, \eqref{I21;22} in order to deduce \eqref{lem:Var:u:t-2}.
	  \end{proof}

\begin{proposition}\label{pr:Var:u:t}
We assume that $u_0$ satisfies the hypotheses of Proposition \ref{pr:heat}. Then,
	\[
		\adjustlimits\sup_{t\in(0,T]}\sup_{x\in\R^d}
		\| u(t+\varepsilon\,,x) - u(t\,,x) \|_k \lesssim  \varepsilon^{\frac12(\eta \wedge \indH)+ o(1)}
		\qquad\text{as $\varepsilon\downarrow0$},
	\]
where $k\ge 2$ and $\eta\in(0,1]$ are given in \eqref{u_0}.	
\end{proposition}

\begin{proof}
Fix $\varepsilon\in(0,1)$. Thanks to \eqref{mild:heat} 
	we may write, for all $x\in\R^d$, $t\in(0\,,T]$ and $\varepsilon>0$, the decomposition:
	\begin{equation}\label{u=I1+I2+I3}
		u(t+\varepsilon\,,x) - u(t\,,x) = J(t\,,x\,;\varepsilon) + [I(t+\varepsilon,x)-I(t,x)] +[H(t+\varepsilon\,,x)-H(t\,,x)],
	\end{equation}
	where $J(t\,,x\,;\varepsilon)= (p(t+\varepsilon) * u_0)(x) - (p(t)*u_0)(x)$.
	
	Proposition \ref{pr:heat} ensures that 
	\begin{equation}\label{I1}
	\|J(t\,,x\,;\varepsilon)\|_k\le c^{1/k} \varepsilon^{\eta/2},
	\end{equation}
uniformly for every $t>0$, $x\in\R^d$.
Furthermore, we recall that	
\begin{equation}
\label{heat:moment-time}
		\sup_{t\in[0,T]}\sup_{x\in\R^d}
		\E\left(\|H(t+\varepsilon\,,x)-H(t\,,x)\|_k\right)\lesssim \varepsilon^{a/2},
		\end{equation}	
for any  $a\in(0\,,\indH)$, uniformly in $\varepsilon\in(0,1)$ ( see \eqref{moment:n}).

Set
\[
f_\varepsilon(t) := \sup_{x\in\R^d} \|u(t+\varepsilon\,,x) - u(t\,,x)\|_k,
	 \qquad 0<t\le T.
	 \]
As was mentioned before, the assumptions of the proposition imply the validity of \eqref{lem:Var:u:t-1}.
Then from the estimates \eqref{I1}, \eqref{heat:moment-time} together with Lemma \ref{lem:Var:u:t}, we see that for any $\varepsilon\in(0,1)$ and 
for any $\delta\in(0,\frac12(\eta \wedge \indH))$,
\[
f_\varepsilon(t)\lesssim \varepsilon^\delta + \int_0^t f_\varepsilon(s)\ \d s,
\]
uniformly in $\varepsilon\in(0,1)$.
The proposition follows from Gronwall's lemma since  $\delta\in(0,\frac12(\eta \wedge \indH))$ can be otherwise arbitrary.
\end{proof}
\smallskip

\begin{proposition}\label{pr:Var:u:x}
The initial condition $u_0$ is as in Proposition \ref{pr:heat}. Then,
	\[
		\adjustlimits\sup_{t\in(0,T]}\sup_{x\in\R^d}
		\| u(t\,,x+h) - u(t\,,x) \|_k \le \|h\|^{\eta\wedge\indL + o(1)}
		\qquad\text{as $h\to 0$},
	\]
	where $k\ge 2$ and $\eta\in(0,1]$ are given in \eqref{u_0}.
\end{proposition}

\begin{proof}
	The proof is similar to that of Proposition \ref{pr:Var:u:t}, but simpler.
	Write
	\[
		u(t\,,x+h) - u(t\,,x) = J_1(t\,,x\,;h) + J_2(t\,,x\,,;h) + [H(t\,,x+h) - H(t\,,x)],
	\]
	where
	\begin{align*}
		J_1(t\,,x\,;h) &= (p(t)*u_0)(x+h) - (p(t)*u_0)(x),\\
		J_2(t\,,x\,;h) &= \int_0^t\d s\int_{\R^d} p(t\,,\d y) 
			\left[g(s\,,y\,,u(s\,,x+h-y)) - g(s\,,y\,,u(s\,,x-y))\right].
	\end{align*}
	Proposition \ref{pr:heat} ensures that
	$\|J_1(t\,,x\,;h)\|_k \lesssim \|h\|^{\eta}$, and since $g$ is uniformly global Lipschitz continuous, applying Minkowski's inequality we deduce
		\[
		\|J_2(t\,,x\,;h)\|_k  \lesssim \int_0^t \sup_{w\in\R^d} \| u(s\,,w+h) - u(s\,,w)\|_k\,\d s,
	\]
	all valid uniformly for all $x,h\in\R^d$ and $t\in(0\,,T]$. Set
	\[
		\bar f_h(t) = \sup_{x\in\R^d} \|u(t\,,x+h)-u(t\,,x)\|_k
		\qquad 0<t\le T.
	\]
	From the above discussion we see that
		\begin{equation}
		\label{h1}
		\bar f_h(t) \le ch^{\eta} + c\int_0^t \bar f_h(s)\,\d s + \sup_{s\in(0,T]}
		\|H(s\,,x+h) - H(s\,,x)\|_k
		\hskip1in(0\le t\le T).
	\end{equation}
	Because of  the estimate \eqref{heat:moment:00} and the identity 
	$\cI_R(T)=2 \indL$ (see Lemma \ref{lem:ind:H}), for every $\gamma\in(0\,, \indL)$,
	\[
		\sup_{t\in[0,T]}\sup_{x\in\R^d}\|H(s\,,x+h) - H(s\,,x)\|_k\lesssim |h|^{\gamma},
	\]
	uniformly for all $h\in\R^d$ that satisfy $\|h\|\le 1$. Consequently,  \eqref{h1} yields
	\[
		\bar f_h(t) \le c_1 h^{\eta\wedge \indL} + c_2 \int_0^t \bar f_h(s)\,\d s\quad 
		\text{for every}\quad t\in(0\,,T].
	\]	 
	The proposition follows from Gronwall's lemma.\end{proof}
	
Assume that the hypotheses of Proposition \ref{pr:heat} on moments of $u_0$ hold for any $k\ge 2$. Then 
	Propositions \ref{pr:Var:u:t} and \ref{pr:Var:u:x} and Kolmogorov's continuity theorem
	imply the following.
	\begin{theorem}\label{th:heat-nonlinear}
	\begin{enumerate}
	\item If $\indH>0$ then for every $x\in\R^d$ , $u(\cdot\,,x)$ is a.s.\ in $C^\alpha_{\text{\it loc}}(\R_+)$
	for any $\alpha\in(0\,,\frac12(\eta\wedge\indH))$.
	\item  If $\indL>0$ then with probability one $u\in C^{\alpha,\beta}_{\text{\it loc}}(\R_+\times\R^d)$
	for every $\alpha\in(0\,,\frac12(\eta\wedge\indH))$ and $\beta\in(0\,, \eta\wedge\indL)$.
	\end{enumerate}
	
	When $u_0\equiv g\equiv0$, this becomes assertions \eqref{H:x} and \eqref{H:t} 
	of Theorem \ref{th:Holder:H}, respectively.
\end{theorem}

It would be interesting to have a counterpart of Theorem \ref{th:heat-nonlinear} for the stochastic wave equation. Unfortunately, this question seems at the moment out of reach. The last part of the section is devoted to describe some of the problems to be solved to progress in that direction.  

Consider the stochastic nonlinear wave equation:
\begin{equation}\label{SWE:NW}
\left[\begin{split}
	&\partial_t^2 u  = \rL u  + g(u) + \dot{\rf}\quad\text{on $(0\,,\infty)\times\R^d$},\\
	&\text{subject to}\quad u(0)=\partial_t u(0+)=0\quad\text{on $\R^d$},
\end{split}\right.\end{equation}
where $g:\R\rightarrow \R$ is a Lipschitz continuous function, 
and we assume that the characteristic exponent $\psi$ satisfies \eqref{RL} and $\psi(\xi)=\psi(-\xi)$. 
Similarly to the case of the heat equation, a stochastic process 
$\{u(t\,,x);\ (t\,,x)\in\R_+\times \R^d\}$ is called a random-field solution to \eqref{SWE:NW} if it satisfies some appropriate measurability conditions and for every $(t\,,x)\in\R_+\times \R^d$,
\begin{equation}\label{s6.1}
	u(t\,,x) = \int_0^t \d s \int_{\R^d} \d y\  G(t-s\,,x-y)\ g(u(s\,,y)) + W(t\,,x),
\end{equation}
a.s., where $G(t\,,\cdot)$ denotes here the fundamental solution of the wave operator $\partial_t^2 - \rL$, and $W$ is the random-field solution to \eqref{SWE:G}. 
Remember that $W$ is defined by the stochastic convolution $G * F$, and also that
\begin{equation}
\label{s6.2}
	\hat G(t\,,\xi) = t\sinc\left( t\sqrt{\psi(\xi)}\right)\quad
	\text{for $t\ge 0,\ \xi\in\R^d.$}
\end{equation}
When $\psi(\xi) = \|\xi\|^2$, and for any spatial dimension $d\ge 1$, 
the pathwise integral $\int_0^t \d s \int_{\R^d}\d y\  G(t-s\,,x-y)\ g(u(s\,,y))$ 
in \eqref{s6.1} is rigorously defined in 
 \cite[Proposition 3.4]{Conus-Dalang}. A further analysis shows that the validity of this proposition can be extended to symmetric characteristic exponents $\psi$ (as described above) if the fundamental solution $G$ satisfies the following:
 \begin{description}
 \item{(a)} $G(t\,,\cdot)\in \mathcal{O}_c^\prime$ (the space of Schwartz distributions of rapid decrease); 
 \item{(b)} the function $(s\,,\xi)\mapsto \hat G(s\,,\xi)$ is measurable;
\item {(c)} $\int_0^T \sup_{\xi\in\R^d} | \hat G(s\,,\xi)|^2\, \d s < \infty$;
\item {(d)} $\lim_{h\downarrow 0} \int_0^T 
	\sup_{\xi\in\R^d} \sup_{s<r<s+h}| \hat G(r\,,\xi)  -  \hat G(s\,,\xi)|^2\,\d s = 0$.
 \end{description}
 This is indeed the case, as we now argue. 
 
According to \cite[Th\'eor\`eme IX, p.244]{Schwartz}, condition (a) is equivalent to saying that 
$G(t\,,\cdot) * \varphi\in\mathcal{S}(\R^d)$
for every   $\varphi\in\mathcal{D}(\R^d)$.  Since $\hat \varphi\in \mathcal{S}(\R^d)$, this follows from the lower bound  $t\sinc( t\sqrt{\psi(\xi)})\gtrsim [1+\psi(\xi)]^{-1}$ along with the assumption \eqref{RL}.

Assertion (b) is a consequence of the continuity of $\psi$. Property (c) follows from the upper bound 
$t\sinc\left( t\sqrt{\psi(\xi)}\right) \lesssim[1+\psi(\xi)]^{-1}$. As regards (d), using \eqref{s6.2}, we have
\begin{equation*}
|\hat G(r\,,\xi) - \hat G(s\,,\xi)| \le \frac{2\wedge |r-s|\sqrt{\psi(\xi)}}{\sqrt{\psi(\xi)}} \le |r-s|
\end{equation*}
(see stage 3 in the proof of Lemma \ref{lem:ind:W}). This yields 
$\sup_{\xi\in\R^d} \sup_{s<r<s+h}| \hat G(r\,,\xi)  -  \hat G(s\,,\xi)|\le h$
and therefore, (d) holds.

We deduce that the hypotheses required in the application of \cite[Theorem 4.2 ]{Conus-Dalang} to the particular case when the coefficient $\alpha$
there vanishes and $\beta\equiv g$, are satisfied. This yields at once two facts: 

1)\ The expression \eqref{s6.1} is well-defined, that means, the integral $\int_0^t \d s \int_{\R^d}\d y\  G(t-s\,,x-\d y)\ g(u(s\,,y))$ is a $L^2$--random variable. In fact, setting $Z(s\,,x):=g(u(s\,,y))$, one has
\begin{align*}
	\left\Vert\int_0^t \d s \int_{\R^d} G(t-s\,,x-\d y)\ g(u(s\,,y))\right\Vert_2^2
	\lesssim \left(\sup_{0\le s\le t}\sup_{x\in\R^d} \Vert Z(s\,,x)\Vert_2^2\right) 
	\int_0^t \d s \sup_{\xi\in\R^d} | \hat G(s\,,\xi)|^2.
\end{align*}

2)\  There exists a stochastic process $\{u(t\,,x);\ (t\,,x)\in\R_+\times \R^d\}$ that satisfies \eqref{s6.1} and
\begin{equation}
\label{s6.4}
	\sup_{0\le s\le t}\sup_{x\in\R^d} \Vert u(s\,,x)\Vert_2 < \infty.
\end{equation}

However, even in the case $\psi(\xi)=\Vert \xi\Vert^2$, the extension of \eqref{s6.4} to $L^p$ moments, with $p>2$, is an open question. Recall that, in Proposition \ref{pr:Var:u:x} (relative to the stochastic heat equation), the existence of $L^p$ moments of any order $p\ge 2$, is fundamental for the proof.  

Comparing with the nonlinear heat equation  \eqref{mild:heat}, we 
see that the drift $g$ in \eqref{s6.1} is more particular and moreover, the initial conditions are null. This is to ensure 
a stationary-type property which is crucial in the theory developed in \cite{Conus-Dalang}, and as a consequence, to give a rigorous meaning to \eqref{SWE:NW}. It is an open and rather speculative project to build an integration theory \`a la Conus-Dalang without stationary constraints on the integrands.

\section{A comment on fractional powers of $\rL$}\label{sec:frac}
Suppose $X$ is symmetric and that there exists $q>0$ such that
\begin{equation}\label{stable_}
	\psi(\xi) \gtrsim \|\xi\|^q\qquad\text{uniformly for all $\xi\in\R^d$}.
\end{equation}
We can deduce from \eqref{Bochner} that $q\le 2$. Property \eqref{stable_} says
that, in some sense, the law of $X(1)$ is smoother than the law of a radially
symmetric $q$-stable process. Under these conditions, 
\[
	\int_{\R^d}\frac{\|\xi\|^{2b}\,\mu(\d\xi)}{1+\psi(\xi)}
	\lesssim\int_{\R^d}\frac{|\psi(\xi)|^{2b/q}\,\mu(\d\xi)}{1+\psi(\xi)},
\]
which, together with \eqref{indR<indH}, yields
\begin{equation}\label{bbb}
	\indL>0\quad\Longleftrightarrow\quad\indH>0
	\quad\Longleftrightarrow\quad\int_{\R^d}\frac{\mu(\d\xi)}{1+[\psi(\xi)]^\tau}<\infty
	\text{ for some $\tau\in(0\,,1)$}.
\end{equation}
Whenever $\tau\in(0\,,1)$, we can introduce an independent $\tau$-stable subordinator
$S$ (see Bertoin \cite{Bertoin}); that is $S$ is a L\'evy process with $S(0)=0$ and nondecreasing
sample functions, normalized such that
$\E\exp(-\lambda S(1))=\exp(-\lambda^\tau)$ for every $\lambda\ge0$. Define
$X'(t) = X(S(t))$ to be the process $X$, subordinated by $S$. Then, 
disintegration shows that $X'$ is a symmetric L\'evy process with 
$\E\exp(i\xi\cdot X'(1)) = \exp( - [\psi(\xi)]^\tau)$ for  all $\xi\in\R^d$.
Moreover, the generator of $X'$ is the $\tau$-th power of $\rL$ -- denoted
by $-(-\rL)^\tau$ -- whose multiplier is defined via symbolic calculus as $-\psi^\tau$.
In the present setting, conditions \eqref{cond:H} and \eqref{cond:R}
are equivalent to one another, as well as the final condition in \eqref{bbb}.  Now we observe
that the final condition in \eqref{bbb} is simply Dalang's condition \eqref{Dalang}
for the L\'evy process $X'$. In this way we can see from Dalang's theory \cite{Dalang}
and Theorems \ref{th:heat} and \ref{th:Holder:H} that, in the present setting
where $\rL$ is self adjoint and satisfies \eqref{stable_},
the optimal condition for the H\"older
regularity condition of the solution to \eqref{SHE} is that the following SPDE has a random field
solution for some $\tau\in(0\,,1)$:
$$\left[\begin{split}
	&\partial_t v = -(-\rL)^\tau v  + \dot{\rf}
		\quad\text{on $(0\,,\infty)\times\R^d$},\\
	&\text{subject to } v(0)=u_0\quad\text{ on $\R^d$},
\end{split}\right.$$
Analogously,  the optimal condition for H\"older regularity of the solution to
\eqref{SWE} is equivalent to the existence of 
random-field solution of the following initial-value problem for some $\tau\in(0\,,1)$:
$$\left[\begin{split}
	&\partial^2_{t} v = -(-\rL)^\tau v  + \dot{\rf}
		\quad\text{on $(0\,,\infty)\times\R^d$},\\
	&\text{subject to } v(0)=\partial_t(v(0))=0\quad\text{ on $\R^d$}.
\end{split}\right.$$

\appendix
\section{Appendix}

\begin{lemma}\label{lem:A}
	If $0<c<b$, then uniformly for all $a>0,$
	\[
		\int_0^1\left( (a\varepsilon)^b\wedge 1\right)\frac{\d\varepsilon}{\varepsilon^{1+c}}
		\asymp\begin{cases}
			 a^b\wedge a^c & \text{if $0<c<b$},\\
			\infty&\text{if $c\ge b$}.
		\end{cases}
	\]
\end{lemma}

\begin{lemma}\label{lem:B}
	If $T>0$ is fixed, then
	$\int_0^T s^{-a}\exp(-2sb)\,\d s \asymp (1+b)^{-1+a}$,
	uniformly for every $a\in(0\,,1)$ and $b\ge0$.
\end{lemma}

\begin{lemma}\label{lem:index}
	If $f:(0\,,1]\to\R_+$ is measurable, then
	\[
		\sup\left\{ b>0:\, \int_0^1 f(t)\,\frac{\d t}{t^{1+b}}<\infty\right\}
		\ge \sup\left\{ a>0:\,
		\limsup_{t\to0^+}\frac{f(t)}{t^a}<\infty\right\}.
	\]
	where $\sup\varnothing=0$. If $f$ is non decreasing and  measurable, then
	we can replace ``$\ge$'' with an identity.
\end{lemma}

\begin{proof}
	First, consider the case that $f$ is non decreasing and measurable.
	Let $I$ denote the quantity on the left-hand side of the identity of the lemma,
	and $J$ the quantity on the right.
	The integral test from calculus tells us that the integral in the definition of 
	$I$ is finite iff $S_b = 
	\sum_{n=1}^\infty \e^{-b n}f(\e^{-n})$ is finite. If $c\in[0\,,I)$,
	then $S_c<\infty$ and hence $\e^{-c n}f(\e^{-n})\to0$
	as $n$ tends to infinity along integers. Apply monotonicity to see that
	$\lim_{t\to0^+} t^{-c}f(t)=0$. This shows that every 
	$c<I$ satisfies $c\le J$, which is to say that $I\le J$. It remains to prove that $J\le I$.
	
	On one hand, if $J=0$ then the above proves that $I=J=0$.
	On the other hand, if $J>0$ and $c\in(0\,,J)$, then 
	$\limsup_{t\to0^+} t^{-c}f(t)<\infty$.
	This is equivalent to saying that
	$c=\sup_{n\in\N}[\e^{c n}f(\e^{-n})]<\infty$, which in turn implies that 
	$S_b\le c\sum_{n=1}^\infty \e^{-n(b-c)}<\infty$
	whenever $b>c$. This proves that every $c<J$ satisfies $c\le I$,
	 whence $J\le I$, and completes the proof of the identity $I=J$.
	 
	 Finally, if $f$ is not necessarily montone, then we observe that
	 $f \le \bar{f}$ where $\bar{f}(t) = \sup_{s\in[0,t]}f(s)$.
	 If $b>I$, then
	 $\int_0^1f(t) t^{-1-b}\,\d t = \infty$
	 and hence $\int_0^1\bar{f}(t) t^{-1-b}\,\d t = \infty$.
	 We may apply the already-proved portion of the lemma to conclude that
	 $\limsup_{t\to0^+} t^{-b}\bar{f}(t)=\infty$, which is another way 
	 to say that $\lim_{t\to 0^+}t^{-b}f(t)=\infty$, and hence
	 $b\ge J$. Since this is true for every $b>I$, it follows that
	 $I\ge J$.
\end{proof}

\begin{small}

\end{small}
\bigskip

\begin{small}
\noindent {\bf Davar Khoshnevisan.}\\
	Department of Mathematics, University of Utah,
	Salt Lake City, UT 84112-0090, USA,
	\texttt{davar@math.utah.edu}\\
	
\noindent {\bf Marta Sanz-Sol\'e.}\\
	Facultat de Matem\`atiques i Inform\`atica, 
	Universitat de Barcelona, Gran Via de les Corts Catalanes, 585, 08007 Barcelona, Spain,
	\texttt{marta.sanz@ub.edu}
\end{small}

\end{document}